\documentclass[11pt,reqno]{amsart}
\usepackage{xifthen}
\usepackage{pkgfile}
\usepackage{multirow,rotating,lscape,mathtools}
\usepackage{subfig}

\def\newF{R}
\def\X{{\mathcal X}}

\def\newC{{\mathcal C}} 
\def\newF{R}
\def\X{{\mathcal X}}
\def\calJ{{\mathcal J}}

\newcommand{\cons}{c^N}
\newcommand{\ucons}{u^N}
\newcommand{\estimate}{\hat{\beta}^N(\vec \sigma)}

\allowdisplaybreaks

\title[Parameter Estimation for Undirected Graphical Models with Hard Constraints]{Parameter Estimation for Undirected Graphical Models with Hard Constraints}

\author[Bhattacharya]{Bhaswar B. Bhattacharya}
\address{Department of Statistics, University of Pennsylvania, Philadelphia, PA 19104, USA.}
\email{bhaswar@wharton.upenn.edu}

\author[Ramanan]{Kavita Ramanan}
\address{Division of Applied Mathematics, Brown University, Providence, RI 02912, USA.} 
\email{kavita\_ramanan@brown.edu}

\begin{document}

\begin{abstract}
  The hardcore model on a graph $G$ with parameter $\lambda > 0$ is a  probability measure on the collection of all independent sets of $G$, that assigns to each independent set $I$ a probability proportional to $\lambda^{|I|}$. In this paper we consider the problem of estimating the  parameter $\lambda$ given a single sample from the hardcore model on a graph $G$. To bypass the computational intractability of the  maximum likelihood method, we use the maximum pseudo-likelihood (MPL) estimator, which for the hardcore model has a  surprisingly simple closed form expression. We show that  for any sequence of graphs $\{G_N\}_{N \geq 1}$, where $G_N$ is a graph on $N$ vertices, the MPL estimate of $\lambda$ is $\sqrt N$-consistent (that is, it converges to the true parameter at rate $1/\sqrt N$), whenever the graph sequence has uniformly bounded average degree. We then extend our methods to obtain estimates for the vector of activity parameters in general $H$-coloring models, in which restrictions between adjacent colors are encoded by a constraint graph $H$.  These constitute an important class of Markov random fields that includes all hard-constraint models, which arise  in a broad array of fields including combinatorics, statistical physics, and communication networks. Given a single sample from an $H$-coloring model,  we derive sufficient conditions under which the MPL estimate
  is $\sqrt N$-consistent.  Moreover, we verify the sufficient conditions for  $H$-coloring models for which  
   there is at least one `unconstrained' color (that is, there exists at least one vertex in the constraint graph $H$ that is connected to all vertices), as long as the graph sequence has uniformly bounded average degree. This applies to many $H$-coloring examples such as the Widom-Rowlinson and multi-state hard-core models. On the other hand, for the $q$-coloring model, which falls outside this class, we show that the condition can fail and consistent estimation may be impossible even for graphs with bounded average degree. Nevertheless, we show that the MPL estimate is $\sqrt N$-consistent in the $q$-coloring model when $\{G_N\}_{N \geq 1}$ has bounded average double neighborhood. The presence of hard constraints, as opposed to soft constraints,  leads to new challenges, and our proofs entail applications of the method of exchangeable pairs as well as combinatorial arguments that employ the probabilistic method. 
\end{abstract}



\maketitle

\section{Introduction}

\subsection{Problem Description and  Motivation} 
  
With the ubiquitous presence of network data in modern statistics it has become increasingly important to develop realistic and mathematically tractable models for dependent and structured high-dimensional distributions. Markov random fields (undirected graphical models) are useful primitives for modeling such datasets, which arise naturally in spatial statistics, social networks, image processing, neural networks, and protein folding, among others. This necessitates developing computationally tractable algorithms for fitting these models to data and understanding their statistical efficiencies (rate of convergence).  However, estimating the parameters in such models using the standard maximum likelihood (ML) estimation method is, in general, notoriously hard due to the appearance of an intractable normalizing constant in the likelihood. One approach to circumvent this issue that has turned out to be particularly useful in various cases is the maximum pseudo-likelihood (MPL) estimator \cite{besag_lattice,besag_nl}, which avoids computing the partition function by maximizing an approximation to the likelihood function (a `pseudo-likelihood') based on conditional distributions.

While it is of course classical that if one had multiple independent and identically distributed (i.i.d.) samples from the underlying model, then the standard likelihood-based estimates for the parameters of the model would be consistent, it is {\em a priori} not clear if one could construct
a consistent estimator given less information. Here, we are interested in the problem of estimating the parameters of a parameterized family of Markov random fields  defined on a graph $G_N$ on $N$ vertices, given a single sample $\vec \sigma=(\sigma_1, \sigma_2, \ldots, \sigma_N)$ from the model. For the Ising model, which corresponds to a Markov random field with a binary outcome and a quadratic sufficient statistic, statistical properties of the pseudo-likelihood estimate were first studied in the seminal paper of Chatterjee \cite{chatterjee},  and later extended in \cite{BM16} and \cite{pg_sm} to include Ising models on general weighted graphs and joint estimation of parameters, respectively. These techniques were recently  used by Daskalakis et al. \cite{cd_ising_I,cd_ising_II} to obtain rates of convergence of the MPL estimates in general logistic regression models with dependent observations. More recently, Dagan et al. \cite{cd_ising_estimation} considered the problem of parameter estimation in a more general Ising model and, as a consequence, improved some of the results in \cite{BM16}. 
Related problems in hypothesis testing given a single sample from the Ising model are considered in \cite{gb_testing,rm_sm,rm_gr}.

All the results mentioned above only consider models with {\it soft constraints}, where every configuration between neighboring vertices has a finite (positive or negative) potential. On the other hand there are many interesting models that    
are subject to {\it hard constraints} that  forbid certain spin configurations between neighboring nodes, as in the hardcore (independent set) model and in graph-coloring problems. These models arise in a variety of fields, including communication network
  applications such as spectrum or channel assignment for wireless networks, wavelength assignment in  WDM (wavelength-division-multiplexing) optical networks, carrier-sense multiple access (CSMA) networks, and networks with multicasting (see, for example, \cite{packing,c,coloring_sensor,d,e,nonmonotone_network,mrf_networks,a,wireless} for a highly incomplete sample of works in this area).  The equilibrium behavior in these systems can in many cases be modeled as Markov random fields with hard constraints \cite{c,e,wireless}, and estimating the parameters of the model from one sample correspond to estimation of demand or arrival rates from a snapshot of the configuration of the network. While there are a few results on consistency for (continuous) hardcore models on spatial point processes \cite{consistent_point_process,spatial_asymptotic_distribution}, estimation problems in the discrete case, especially when the underlying graph is not a lattice or when the parameters lie outside the `high temperature' regime, have remained largely unexplored.

\subsection{Discussion of Our Contributions}
In this paper, we initiate the study of parameter estimation in general $H$-coloring models, for some constraint graph $H$, which is an important class of Markov random fields that includes all hard-constraint models, including independent sets (equivalently, the hard-core model), the Widom-Rowlinson model, and proper colorings \cite{bhw,bw,gmrt,georgii_book,wr}.  Using the method of exchangeable pairs developed in \cite{chatterjee_thesis,chatterjee}, we first show that the MPL estimate of the fugacity parameter $\lambda$ in the classical hardcore (independent set) model is $\sqrt N$-consistent (that is, the estimate converges to the true parameter at rate $1/\sqrt N$), for any sequence of graphs $G_N$ on $N$ vertices with  uniformly bounded {\it average} degree, that is, $\sup_{N \geq 1} \frac{1}{N} \sum_{u=1}^N d_u(G_N) < \infty$, where $d_u(G_N)$ denotes the degree of the vertex $u$ in $G_N$ (see Theorem \ref{thm:beta_estimate} for the formal statement). The exchangeable pairs approach allows us to establish the concentration of the derivative of the log-pseudolikelihood function without using any decay of correlation estimates (which are only available in the high-temperature/Dobrushin uniqueness regime), and, as a consequence, we can obtain the rate of consistency of the pseudo-likelihood estimate for all temperatures.

Next, we study the estimation problem for general $H$-coloring models where, unlike in the hardcore model, bounded average degree does not necessarily guarantee consistent estimation (Section \ref{subs-eg1}). We derive conditions under which the MPL estimate is $\sqrt N$-consistent in general $H$-coloring models (see Theorem \ref{thm:beta_H}) and discuss how it can be efficiently computed using a simple gradient descent algorithm.  As a consequence, we are able to show $\sqrt{N}$-consistency of  MPL estimates in $H$-coloring models for which the constraint graph $H$ has at least one vertex that is connected to all vertices, that is, it has an `unconstrained' color, whenever $\{G_N\}_{N \geq 1}$ has uniformly bounded average degree (Corollary \ref{cor:unconstrained_H}). In particular, this shows $\sqrt N$-consistency of the MPL estimates for the multi-state hardcore and the Widom-Rowlinson models, since in both these models the color or state 0 is unconstrained.  Finally, we consider the $q$-coloring model, which is an $H$-coloring model where no vertex is unconstrained. In this case, unlike in models with unconstrained colors, consistent estimation might be impossible for graphs with bounded average degree (as we show in Section \ref{subs-eg2}). Therefore, one needs to go beyond  the boundedness of the average degree (equivalently, size of the 1-neighborhood) for obtaining consistent estimates in the $q$-coloring model. Interestingly, a condition that works is the boundedness of the average size of the 2-neighborhood (the 2-neighborhood of a vertex $v \in G_N$ is the set of vertices that are at a distance of either 1 or 2 from $v$ in $G_N$). More precisely, we show in  Corollary \ref{cor:coloring} that the MPL estimate is $\sqrt N$-consistent in the $q$-coloring model whenever the graph sequence $\{G_N\}_{N \geq 1}$ has uniformly bounded average 2-neighborhood.  One of the technical highlights of this proof is a  probabilistic method argument, inspired by the proof of the celebrated Turan's theorem from extremal combinatorics \cite[page 95]{independent_set_book}, which shows that any graph with a bounded average 2-neighborhood has a linear size subset with mutually disjoint neighborhoods.  

\subsection{Related Work on Structure Learning with Multiple Samples}

Another related area of active research is the problem of structure learning in Markov Random Fields. Here, one is given access to {\it multiple} i.i.d.~samples from a general graphical model, such as the Ising model or the hardcore model, and the goal is to estimate the underlying graph structure. This problem has been extensively studied for models with soft constraints, especially the Ising model (cf. \cite{structure_learning,bresler,graphical_models_algorithmic,highdim_ising,graphical_models_binary} and the references therein).
For the hardcore model \eqref{eq:lambda_I}, the structure learning problem was first studied by Bresler et al. \cite{antiferromagnetic}. They proposed an algorithm for learning the underlying network, which required only logarithmic samples in the number of nodes, for bounded degree graphs. Extensions to general $H$-coloring models and precise conditions for identifiably was obtained in \cite{coloring_structure}.  The related problem of identity testing given multiple samples from a ferromagnetic Ising model was considered by Daskalakis et al. \cite{cd_testing}. For anti-ferromagnetic Ising models and models with hard constraints, in particular for proper colorings, this problem has been recently studied in \cite{coloring_testing}.

All these results, however, are in contrast with the present work, where the underlying graph structure is assumed to be known and the goal is to  estimate the natural parameters given a {\it single} sample from the model. This is motivated by applications where it is more common to have access to only a single sample on the whole network, because  it is often difficult to generate many independent samples from the underlying model within a reasonable amount of time.

\subsection{Asymptotic Notation} For positive sequences $\{a_n\}_{n\geq 1}$ and $\{b_n\}_{n\geq 1}$, $a_n = O(b_n)$ means $a_n \leq C_1 b_n$ and $a_n =\Theta(b_n)$ means $C_2 b_n \leq a_n \leq C_1 b_n$, for all $n$ large enough and positive constants $C_1, C_2$. Similarly, for positive sequences $\{a_n\}_{n\geq 1}$ and $\{b_n\}_{n\geq 1}$, $a_n \lesssim b_n$ means $a_n \leq C_1 b_n$ and $a_n \gtrsim b_n$ means $a_n \geq C_2 b_n$ for all $n$ large enough and positive constants $C_1, C_2$. Moreover, subscripts in the above notation,  for example $O_\square$, $\lesssim_\square$, $\gtrsim_\square$, and $\Theta_\square$,  denote that the hidden constants may depend on the subscripted parameters.

\section{Statements of Main Results}

As  mentioned earlier, to estimate parameters of families of Markov random fields,
  we will use the maximum pseudo-likelihood (MPL) estimator, introduced by Besag \cite{besag_lattice,besag_nl}, which we now define.

\begin{defn}{\em (MPL estimate \cite{besag_lattice,besag_nl})}
  \label{def:MPL}
  Given a discrete random vector $\bm X= (X_1, X_2, \ldots, X_N)$ whose joint distribution is parametrized by an $s$-dimensional parameter $\bm \beta \in \R^s$,  the MPL estimate of $\bm \beta$ is defined as
\begin{equation}\label{eq:mple_defn}
\hat {\bm \beta}^N = \hat{\bm \beta}^N (\bm   X) := \arg\max_{\bm \beta \in \R^s}\prod_{i=1}^N f_i(\bm \beta, \bm X),
\end{equation}
where $f_i(\bm \beta, \bm X)$ is the conditional probability mass function of $X_i$ given $(X_j)_{j \ne i}$. 
\end{defn}

 The quantity $\prod_{i=1}^N f_i(\bm \beta, \bm X)$ is referred to as the
    pseudo-likelihood (function), and its logarithm the log pseudo-likelihood (function).
    The MPL estimate is a maximizer of both  the pseudo-likelihood  and the log pseudo-likelihood  functions.

\begin{remark}
  \label{rem-mplabuse}
  {\em   
    With some abuse  of notation,  we will  refer to any critical point of the log pseudo-likelihood
    function as an MPL estimate.  This  will  have no material consequence
    since in  the models we consider,  we will prove that with
    high probability, the log pseudo-likelihood function is asymptotically strictly concave,  thus showing that
    as  $N \rightarrow \infty$, the MPL estimate is  with high probability the unique  critical point  of the log pseudo-likelihood function. 
  }
\end{remark}

\subsection{Estimation in the Hardcore Model}
\label{pl_estimate}

Let $G_N=(V(G_N), E(G_N))$ be a finite simple graph on $N$ vertices, with $V(G_N)=[N]:=\{1, 2, \ldots, N\}$. Denote the adjacency matrix of $G_N$ by $A(G_N) = ((a_{uv}(G_N)))_{1 \leq u, v \leq N}$. For $u \in V(G_N)$, let $\cN_{G_N}(u)$ be the neighborhood of $u$ in $G_N$, that is, $\{v: (u, v) \in E(G_N)\}$, and recall that $d_u(G_N) = |\cN_{G_N}(u)|$ denotes the degree of vertex $u$ in $G_N$.  Throughout we will assume that $G_N$ has no isolated vertex, that is, the minimum degree of any vertex in  $G_N$ is at least 1, for all $N \geq 1$. A set $S \subset V(G_N)$ is said to be an {\it independent set} in $G_N$, if no two vertices in $S$ are adjacent in $G_N$.  The {\it hardcore model} on $G_N$ with {\it fugacity parameter} $\lambda > 0$, is a probability distribution over the collection of independent sets of $G_N$, given by 
\begin{align}\label{eq:lambda_I}
\P^N_\lambda(\vec \sigma) & :=\frac{\lambda^{\sum_{u=1}^N \sigma_u} \prod_{(u, v) \in E(G_N)} \bm 1\{\sigma_u+\sigma_v \leq 1\}}{Z_{G_N}(\lambda)} \nonumber \\ 
& = \frac{\lambda^{\sum_{u=1}^N \sigma_u} \prod_{(u, v) \in E(G_N)} (1- \sigma_u \sigma_v ) }{Z_{G_N}(\lambda)} ,
\end{align}
where $\vec \sigma =(\sigma_1, \sigma_2, \ldots, \sigma_N) \in \{0, 1\}^N$ and $Z_{G_N}(\lambda)$ is the normalization  constant (also referred to as the {\it partition function}) that is determined by the condition $\sum_{\vec \tau\in \{0, 1\}^N}\P^N_\lambda(\vec \sigma= \vec \tau)=1$. Note that $\bm 1\{\sigma_u+\sigma_v \leq 1\}= 0$ if and only if $\sigma_u=\sigma_v=1$, which implies the constraint $\prod_{(u, v) \in E(G_N)} \bm 1\{\sigma_u+\sigma_v \leq 1\}= 1$ is satisfied if  and only if  $\{w: \sigma_w=1\}$ forms an independent set in $G_N$.  Defining $\log 0 := - \infty$ and $e^{-\infty} := 0$, the family of measures $\P^N_\lambda, \lambda > 0,$ in \eqref{eq:lambda_I} can  equivalently be parameterized in terms of $\beta = \log \lambda \in \R$ and
written in exponential form as follows:    
\begin{align}\label{eq:pmf_beta}
\P^N_\beta(\vec \sigma) &=\exp\left\{\beta \sum_{u=1}^N \sigma_u + \log \newC_{G_N}(\vec \sigma) - F_{G_N}(\beta)\right\},
\end{align} 
where the functional $\newC_{G_N}$ that captures the model constraints is given by 
  \begin{align}
    \label{eq:CGn}
    \newC_{G_N}(\vec \sigma) := \prod_{(u, v) \in E(G_N)} (1- \sigma_u \sigma_v ),
  \end{align} 
and $F_{G_N}(\beta) = \log Z_{G_N}(e^\beta)$ is the {\it log-partition function}, that is,
\begin{equation}
F_{G_N}(\beta):=\log\left\{\sum_{\vec \tau\in \{0, 1\}^N}e^{\beta \sum_{u=1}^N \tau_u +  \log \newC_{G_N}(\vec \sigma) }\right\}. 
\label{eq:FGn}
\end{equation}
The  hardcore model arises in diverse fields such as combinatorics, statistical physics, and telecommunication networks. This includes, among others, the study of random independent sets of a graph \cite{bw_bethe,galvin_kahn}, the study of gas molecules on a lattice \cite{exactly_solvable}, and in the analysis of multicasting in telecommunication networks \cite{fk_networks_I,fk_networks_II,mrf_networks}.

In this paper, we first consider the problem of estimating the parameter $\beta$ given a network $G_N$ and a single sample $\vec \sigma$ from the model \eqref{eq:pmf_beta}. As mentioned earlier, it is computationally intractable to use the ML method for estimating $\beta$ because the log-partition function $F_{G_N}(\beta)$ in \eqref{eq:FGn} is generally not computable, analytically or numerically. Even though there are various methods for computing approximate ML estimates  \cite{geyer_thompson}, very little is known about the number of steps required for convergence. Moreover, even if one assumes that the ML estimate has been approximated somehow, general conditions that guarantee consistency of the ML estimate in such discrete exponential families on large combinatorial spaces are not available.

To bypass these obstacles, we will use the maximum pseudo-likelihood (MPL) estimator of
Definition \ref{def:MPL}, which provides a way to conveniently approximate the joint distribution of $\vec \sigma \sim \P^N_\beta$ that avoids calculations with the normalizing constant. We begin by showing that the MPL estimate for the hardcore model takes a fairly explicit form.  To this end,
  fix a finite graph $G_N$ and
  for any allowable hard-core configuration $\vec \sigma \in \{0,1\}^{V_{G_N}}$,  define
  \begin{equation}
    \label{def-aN}
    \cons(\vec \sigma) :=|\{u: \sigma_u=1\}| \end{equation}
   to be the size of the independent set in  the configuration $\vec \sigma$, and define 
   \begin{equation}
     \label{def-uN}
     \ucons(\vec \sigma) := |\cU^N(\vec \sigma)|, \end{equation} 
   where
   \[ \cU^N(\vec \sigma) := \{ u \in V(G_N): \sigma_u=0 \text{ and }\sigma_v=0
   \,  \forall  v \in \cN_{G_N}(u)\} \]
   is the set of vertices assigned zero such that all its neighbors are also assigned zero.
   Note that $\ucons(\vec \sigma)$, the size of the set $\cU^N(\vec \sigma)$, 
 is the number of vertices left `unconstrained' by $\vec \sigma$ in the
sense that for every $u \in \cU^N(\vec \sigma)$, the new spin configuration obtained by
changing the value of $\sigma_u$ to 1 and keeping the other coordinates unchanged, also forms an  independent set in $G_N$.  
Also, for $b \in  \R$,  define 
\begin{align}
 \label{eq:mple_solution}
L_{\vec \sigma}^N (b) & :=\frac{1}{N} \sum_{u=1}^N \left\{ \sigma_u -  \frac{ e^{b + \sum_{v \ne u}  a_{uv}(G_N) \log  (1- \sigma_v)} }{ e^{b + \sum_{v \ne u}  a_{uv}(G_N) \log  (1- \sigma_v)}+ 1 } \right\}. 
\end{align}

\begin{ppn}
  \label{ppn-hardcore}
   Fix $\beta \in \R$, and let $\P^N_{\beta}$  denote the measure
  for the hardcore model with parameter  $\beta$, as defined in \eqref{eq:pmf_beta}. 
  Then for   any sample   $\vec \sigma\sim \P^N_{\beta}$,  $N L_{\vec \sigma}^N(b)$
    is the derivative (with
    respect to $b$)  
    of the log  pseudo-likelihood function. 
  Furthermore, for $b \in \R$, 
  \begin{equation}
       \label{LN-rewrite}
  L_{\vec \sigma}^N (b) = \frac{1}{N} \left( \cons(\vec \sigma) - \frac{e^b}{e^b+1} [ \cons(\vec \sigma)  + \ucons(\vec \sigma)] \right).  
   \end{equation}
and  hence, the   MPL estimate $\hat{\beta}^N = \hat{\beta}^N (\vec \sigma)$ takes the explicit form
  \begin{equation}\label{eq:beta_estimate}
    \estimate =\log \left(\frac{\cons(\vec \sigma)}{\ucons(\vec \sigma)}\right), 
       \end{equation} 
 \end{ppn}   
\begin{proof}  
  To compute the MPL estimate,
  fix $\beta \in \R$ and consider $\vec \sigma\sim \P^N_{\beta}$. Then from  \eqref{eq:pmf_beta}, the conditional probability simplifies to 
\begin{align}
  \label{cond-prob}
\P^N_\beta(\sigma_u|(\sigma_v)_{v \ne u})=\frac{e^{\beta \sigma_u+\sum_{v \ne u} a_{uv}(G_N) \log (1- \sigma_u\sigma_v)}}{e^{\beta +\sum_{v \ne u} a_{uv}(G_N) \log (1- \sigma_v)}+1}. 
\end{align}
Therefore, the MPL estimate of $\beta$ for the hardcore model, as defined in \eqref{eq:pmf_beta}, 
is obtained by considering the critical points (see Remark \ref{rem-mplabuse}), 
of the normalized log pseudo-likelihood function given below:  
\begin{align}\label{eq:pl_beta}
& \frac{1}{N} \log \prod_{u=1}^N \P^N_b(\sigma_u|(\sigma_v)_{v \ne u})  \nonumber \\ 
& = \frac{1}{N} \sum_{u=1}^N \Bigg\{ b\sigma_u+\sum_{v \ne u}  a_{uv}(G_N) \log  (1- \sigma_u\sigma_v) -  \log\left(e^{b +\sum_{v \ne u} a_{uv}(G_N) \log (1- \sigma_v)}+1\right) \Bigg\}. 
\end{align} 
The derivative  (with respect to $b$) of  the right-hand side of  \eqref{eq:pl_beta}  
is easily seen to coincide with the expression for    $L_{\vec \sigma}^N (b)$ in \eqref{eq:mple_solution},
thus proving the first assertion of the proposition.

Next, note that $\sum_{v \ne u}  a_{uv}(G_N) \log  (1- \sigma_v)=0$ if $\sigma_v=0$ for all  $v \in \cN_{G_N}(u)$. Otherwise, $\sum_{v \ne u}  a_{uv}(G_N) \log  (1- \sigma_v)=-\infty$. Thus, we have 
  \begin{align*}
     L_{\vec \sigma}^N (b) =    \frac{1}{N} \Bigg\{ |\{u: \sigma_u=1\}| - \frac{e^b}{e^b+1} |\{ u \in V(G_N): \sigma_v=0 \text{ for all }  v \in \cN_{G_N}(u)\}| \Bigg\}
  \end{align*}
  Now, since $\sigma_u = 1$ automatically implies $\sigma_v = 0$ for all $v \in \cN_{G_N}(u)$, 
  we  obtain \eqref{LN-rewrite}. Combining this with \eqref{eq:mple_solution}, we see that any 
  solution of $L_{\vec \sigma}^N (b)= 0$  satisfies  
   \begin{equation}
     \label{eq:mple_independent}
     \frac{e^{b}}{e^{b} + 1}  =
     \frac{\cons(\vec \sigma)}{\cons(\vec \sigma) + \ucons(\vec \sigma)}, 
   \end{equation}
   from  which it is straightforward to deduce the form 
   \eqref{eq:beta_estimate} for the MPL estimate (keeping in mind  Remark \ref{rem-mplabuse}). 
   \end{proof}

\begin{remark}
  {\em Given \eqref{eq:beta_estimate} and the fact that
  $\P_\beta^N$  is defined in terms of $\P_\lambda^N$ via the reparameterization relation $\lambda = \log \beta$, it immediately
  follows that  the MPL estimate  $\hat \lambda^N$ for the parameter  $\lambda$   of  $\P_\lambda^N$ in \eqref{eq:lambda_I}
  also has an explicit form, now given by
  \begin{equation}\label{eq:lambda_estimate}
    \hat \lambda^N (\vec \sigma) =\frac{\cons(\vec \sigma)}{\ucons(\vec \sigma)}.
  \end{equation}
  }
\end{remark}

Having computed the MPL estimate $\estimate$ for $\beta$ in the hardcore model, it is natural to ask whether it is consistent and, if so, what is the rate of consistency? Note that unlike in classical statistical estimation problems where one has access  to multiple independent samples from the underlying model, our sample consists of a single spin realization from \eqref{eq:pmf_beta}, which has dependent coordinates. This renders standard techniques for proving consistency and central limit theorems inapplicable, and presents a unique challenge where even the most basic characteristics, like correlations between spins at different coordinates are, in general, totally intractable for general graphs. We bypass these hurdles by using the method of exchangeable pairs introduced in \cite{chatterjee_thesis,chatterjee}, to show that the simple MPL estimate of $\beta$ obtained in \eqref{eq:beta_estimate} converges to the true parameter at rate $O(1/\sqrt N)$ whenever $G_N$ has bounded average degree. Recall that $d_u(G_N)$ denotes  the degree of the vertex $u$ in $G_N$.

\begin{thm}\label{thm:beta_estimate} Fix $\beta \in \R$. For $N \in \N$,
    suppose
  $G_N$ is a graph on $N$ vertices,   and the  sequence  
    $\{G_N\}_{N \in \N}$  has uniformly bounded average degree, that is,
  \begin{equation}
    \label{deg-cond}
    \sup_{N \geq 1} \frac{1}{N} \sum_{u=1}^N d_u(G_N) < \infty.
    \end{equation}
  Then, given a sample $\vec \sigma \sim \P^N_\beta$ from the hardcore model on $G_N$ as in \eqref{eq:pmf_beta}, the MPL estimate $\estimate $, as derived in \eqref{eq:beta_estimate}, is $\sqrt N$-consistent for $\beta$, that is, 
  $$
  \lim_{M \rightarrow \infty} \limsup_{N \rightarrow \infty} \P^N_\beta(\sqrt N|\estimate -\beta| > M)= 0.$$ 
\end{thm}

\begin{remark} {\em The result above shows that the MPL estimate is $\sqrt N$-consistent in the hardcore model whenever $\sum_{u=1}^N d_u (G_N) = 2 |E(G_N)|=O(N)$, that is, the number of edges in the graph $G_N$ is asymptotically at most linear in the number of vertices. In particular, this includes graph sequences with uniformly bounded maximum degree and the sparse Erd\H os-R\'enyi random graph sequence $\{{\mathcal G}(N, c/N)\}_{N \in \N}$, for some fixed $c > 0$.} 
\end{remark}

  The proof of Theorem \ref{thm:beta_estimate}, which is given in Section \ref{sec:pfbeta_estimate}, involves the following two steps: 
\begin{itemize}

\item In the first step we use the method of exchangeable pairs to show that the derivative of the normalized log pseudo-likelihood \eqref{eq:mple_solution} is concentrated around zero at the true model parameter, that is, for any $\beta \in \R$, $\E_\beta(L_{\vec \sigma}^N(\beta) ^2) = O(1/N)$ (see Lemma \ref{lm:2moment} for details).

\item Then we show that the log pseudo-likelihood is asymptotically strongly concave, that is, its double derivative is strictly negative for all sufficiently large $N$  when the degree condition
  \eqref{deg-cond} holds. Here, a direct calculation shows that the second derivative of the negative log pseudo-likelihood is proportional to $\cons(\vec \sigma) + \ucons(\vec\sigma) $. Therefore, to show strong concavity it suffices to show that $\cons(\vec \sigma)$ is $\Theta(N)$ with high probability, which, in turn, is proved by using \eqref{deg-cond}  to show that the limiting normalized log partition function  is nowhere flat in the limit (see Lemma \ref{lm:lp}). 
\end{itemize}

\subsection{Estimation in General $H$-Coloring Models}

A realization $\vec \sigma$ from the hardcore model \eqref{eq:lambda_I} can be viewed as a 2-coloring of the vertices
of the graph $G_N$ (with colors labeled 0 and 1), where no two vertices with color 1 are adjacent.
This can be generalized to more than two colors and arbitrary adjacency restrictions between the colors
encoded by a constraint graph $H$. To this end, given a graph $H=(V(H), E(H))$, with vertex set $V(H)=\{0, 1, \ldots, q-1\}$
representing colors (or spins), and possible self loops (but no multiple edges), an $H$-{\it coloring} of $G_N$
is an assignment of colors $\{0, 1, \ldots, q-1\}$ to the vertices of $G_N$ such that adjacent vertices of $G_N$
receive adjacent colors in $H$. More formally, denote the collection of $H$-colorings of $G_N$ by 
\begin{align*} 
\Omega_{G_N}(H) & :=\Big\{\vec \sigma = (\sigma_1, \sigma_2, \ldots, \sigma_N) \in \{ 0, 1, \ldots, q-1\}^N: \nonumber \\ 
& ~~~~~~~~ (u, v) \in E(G_N) \text{ implies }(\sigma_u, \sigma_v) \in E(H) \Big\}.
\end{align*}
The $H$-{\it coloring model} on $G_N$ with {\it activity parameter} $\bm \lambda=(\lambda_1, \lambda_2, \ldots, \lambda_{q-1})' \in \R^{q-1}$ is a probability distribution on $\Omega_{G_N}(H)$ given by,  
$$\P^N_{\bm \lambda}(\vec \sigma) \propto \prod_{u =1}^N \lambda_{\sigma_u}, \quad \text{for } \vec \sigma=(\sigma_1, \sigma_2, \ldots, \sigma_n) \in \Omega_{G_N}(H),$$ 
where $\lambda_0 :=1$. This model was originally introduced in \cite{bw,bw_phase_transitions} and has since  found many interesting applications in statistical physics, combinatorics, and computer science (cf. \cite{bhw,complexity_gh,gmrt,nonmonotone_network,mrf_networks} and references therein).  To write this model more explicitly, for $\vec \sigma \in \Omega_{G_N}(H)$, $u  = 1, \ldots, N$ and $s = 0,  \ldots, q-1$, define
\begin{equation}
  \label{chi}
  \chi_{\vec \sigma, u}(s) := \bm 1\{\sigma_u=s\}
\end{equation} 
and let  
\begin{align}\label{eq:count_sigma}
\cons_s(\vec \sigma) :=\sum_{u=1}^N \chi_{\vec \sigma, u}(s) = \sum_{u=1}^N \bm 1\{\sigma_u=s\}, 
\end{align} 
denote the number of vertices in $G_N$ with color $s$. Denote by $\overline H$ the complement of the graph $H$, that is, $\overline H = (V(\overline H), E(\overline H))$, where $V(\overline H) = V(H)$ and $E(H) = (V(H) \times V(H)) \backslash E(H) $. Denote the adjacency matrix of $H$ by $A(H)=((a_{s t}(H)))_{0 \leq s, t \leq q-1}$.  Then the $H$-coloring model on $G_N$ with parameter 
$\bm \lambda$ is defined as 
\begin{align}\label{eq:model_H}
\P^N_{\bm \lambda}(\vec \sigma)& :=\frac{\prod_{s=1}^{q-1}  \lambda_s^{\cons_s(\vec \sigma)}  \newC_{G_N, H}(\vec \sigma)  }{Z_{G_N}(\bm \lambda)}, 
\end{align} 
where for $\vec \sigma =(\sigma_1, \sigma_2, \ldots, \sigma_N) \in \{0, 1\}^N$, 
\begin{align*}
\newC_{G_N, H}(\vec \sigma) := \prod_{(s, t) \in E(\overline{H})} \prod_{1 \leq u \ne v \leq  N} a_{uv}(G_N) (1-\chi_{\vec \sigma, u}(s) \chi_{\vec \sigma, v}(t)), 
\end{align*} 
and $Z_{G_N}(\bm \lambda)$ is the normalization constant (partition function), which is determined by the condition $\sum_{\vec \tau \in \{0,1, \ldots, q-1\}^N} \P_{\bm \lambda}^N (\vec \tau)= 1$. Setting $\beta_s :=\log \lambda_s$ for $s  = 1, \ldots, q-1$,  and
  $\bm\beta = (\beta_1, \ldots, \beta_s)$,  
the family of distributions in \eqref{eq:model_H} can be reparameterized and written in
exponential form as 
\begin{align}\label{eq:model_H_II}
\P^N_{\bm \beta}(\vec \sigma) = \exp\left\{\sum_{s=1}^{q-1}  \beta_s \cons_s(\vec \sigma) + \log  \newC_{G_N, H}(\vec \sigma)    - F_{G_N}(\bm \beta)\right\} , 
\end{align} 
where  $F_{G_N}(\bm \beta) := \log Z_{G_N}( e^{\bm \beta})$, with the exponential in ${\bm \lambda}  = e^{\bm \beta}$ being taken  coordinatewise. Hereafter, we will assume that the model \eqref{eq:model_H_II} has at least one hard constraint, that is, the complement of the constraint graph $\overline H$ is not empty. In other words, we assume $H \ne K_q^+$, where $K_q^+$ is the complete graph on $V(H)$ with a self-loop at every vertex.

\begin{figure*}[!h]
\centering
\begin{minipage}[c]{1.0\textwidth}
\centering
\includegraphics[width=5.25in]
    {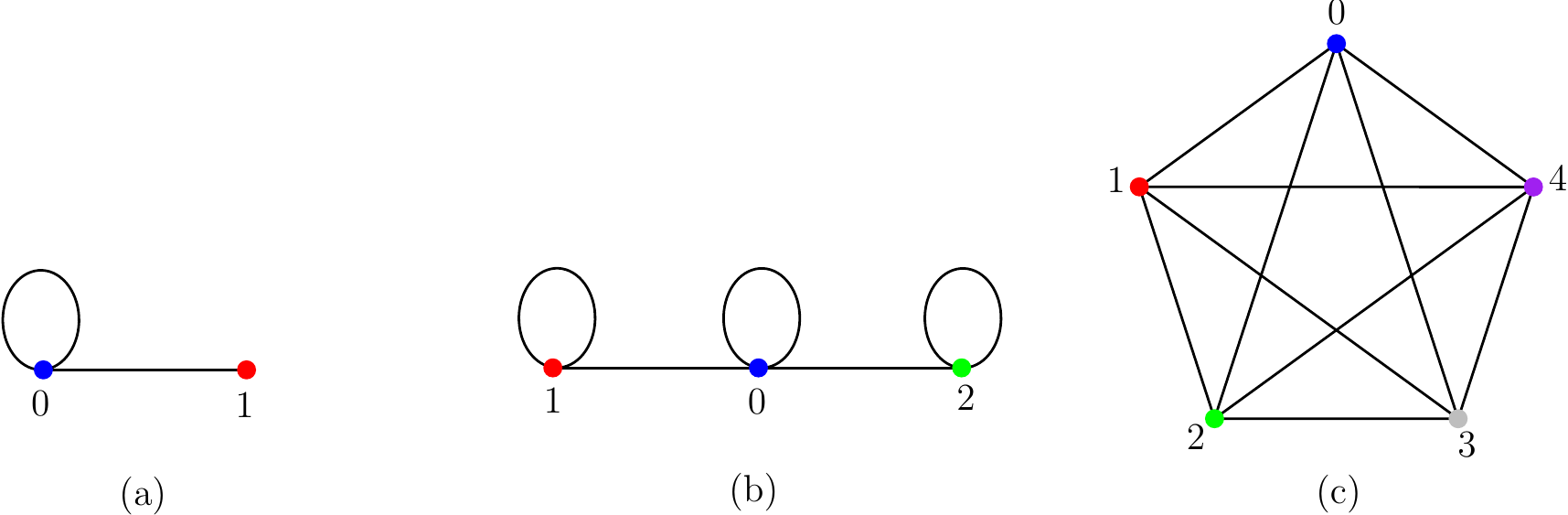}\\
\end{minipage}
\caption{\small{The constraint graphs for the different $H$-coloring models: (a) the hardcore model, (b) the Widom-Rowlinson model, and (c) the $q=5$ proper coloring model.}}
\label{fig:graphs}
\end{figure*}

Many well-known models in statistical physics and combinatorics can be viewed as an $H$-coloring model for some choice of $H$. This includes, among others, the following: 
\begin{itemize}

\item {\it Hardcore and Multi-state Hardcore Models}: The hardcore model \eqref{eq:lambda_I} corresponds to the $H$-coloring model with $H$ equal to the graph on two vertices (labeled 0 and 1) with an edge connecting 0 and 1 and a self-loop at 0,
  as shown in Figure \ref{fig:graphs}(a). Note that two neighboring vertices of $G_N$ are forbidden to have 1 in an $H$-coloring of $G_N$, because there is no self-loop at 1. The multi-state hardcore model \cite{gmrt,two_sided,mrf_networks} is a generalization of this model with $q+1$ states, for $q \geq 1$,  where the constraint graph $H=(V(H), E(H))$, with $V(H)=\{0, 1, \ldots, q\}$, has the edge $(s, t)$ if and only  if   $s+t \leq q$, for $s,t \in V(H)$.

\item {\it Widom-Rowlinson Model}: This is an $H$-coloring model with 3 colors, where the constraint graph $H$ is as shown in Figure \ref{fig:graphs}(b). In this $H$-coloring of $G_N$ a vertex colored 1 is forbidden to be adjacent to a vertex colored 2, since there is no edge between 1 and 2 in $H$. This was introduced in \cite{wr}, as
a model for two types of interacting particles with a hard-core exclusion between particles of
different types (colors 1 and 2 represent particles of each type and the color 0 represents an
unoccupied site), and has been well studied on lattices \cite{wr_discrete}, regular trees \cite{bhw}, and also in the continuum \cite{wr_geometric,wr_phase_transition}.

\item {\it Proper $q$-Coloring Model}: Here, $H=K_q$ is the complete graph on $q$ vertices (as illustrated in Figure \ref{fig:graphs}(c)), that is, $\Omega_{G_N}(H)$ is the collection of all proper colorings of $G_N$ with $q\geq 2$ colors. This is one of the canonical examples of spin systems with hard-constraints that has been widely studied and has found many applications (cf.~\cite{coloring_testing,coloring_structure,bw,bw_phase_transitions,complexity_gh,graph_homomorhism_book,toft_book} and  references therein). 

\end{itemize}

As in the case of the hardcore model,  in Proposition
  \ref{ppn:Hcoloring} below, we first
  establish  a useful characterization of the MPL estimate
  of $\bm \beta$ given a sample  from the $H$-coloring model. 
  First, fix a constraint graph $H$, set $q := |V(H)|$ and a graph $G_N$.
  For any allowable configuration $\vec \sigma \in \Omega_{G_N}(H)$ for the
$H$-coloring model,  consider the mapping
$Q_{G_N, H, \vec \sigma}: V(G_N) \times V(H) \mapsto \{0,-\infty\}$ given 
by 
\begin{align}
  \label{eq:Q_sigma0} 
Q_{G_N, H, \vec \sigma}(u, s) := \sum_{v \in V(G_N)\backslash\{u\}}  \sum_{t=0}^{q-1}  a_{st}(\overline H) a_{uv}(G_N) \log (1- \chi_{\vec \sigma, v}(t)), 
\end{align}
with $\chi_{\vec \sigma, v}$ as defined in \eqref{chi},  and  let 
\begin{align}\label{eq:LNb}
 L_{\vec \sigma}^N(\bm b) := (L^{(N,1)}_{\vec \sigma}(\bm b),  \ldots, L^{(N,q-1)}_{\vec \sigma}(\bm b))', \, \bm b \in \R^q,   
\end{align} 
where, for $r = 1, \ldots, q-1$, $L_{\vec \sigma}^{(r)} = L_{\vec \sigma}^{(N,r)}$ is given by  
\begin{align}
  \label{eq:mple_equation_HI}
L^{(r)}_{\vec \sigma}(\bm b)   := \frac{1}{N} \cons_r(\vec \sigma) - \frac{1}{N} \sum_{u=1}^N \left\{ \frac{  e^{b_r} \bm 1\{Q_{G_N, H, \vec \sigma}(u, r)=0\} }{  \sum_{s=0}^{q-1} e^{b_s} \bm 1\{Q_{G_N, H, \vec \sigma}(u, s)=0\}} \right\},
  \end{align}
with $\cons_r$ the number of vertices with color $r$ in the configuration $\vec \sigma$,
as defined in \eqref{eq:count_sigma}.

Note that $Q_{G_N, H, \vec \sigma}(u, s)  = 0$, if $ \chi_{\vec \sigma, v}(t)= \bm 1\{\sigma_v = t\} =0$, for all $t \in \cN_{\overline H}(s)$ and all $v \in \cN_{G_N}(u)$, and $Q_{G_N, H, \vec \sigma} = -\infty$ otherwise, 
where recall that for any graph $G$, $\cN_G(v)$ denotes the neighborhood of the vertex $v$ in $G$.
In other words, 
\begin{align}\label{eq:Q_sigma}
Q_{G_N, H, \vec \sigma}(u, s) = 
\left\{
\begin{array}{cc}
0  &   \text{ if }  \vec \sigma_{\cN_{G_N}(u)} \subseteq \cN_{H}(s),  \\
- \infty  &    \text{ otherwise, }  
\end{array}
\right.
\end{align}
where $\vec \sigma_A= \{\sigma_a : a \in A\}$.  Thus,   $Q_{G_N, H, \vec \sigma}(u, s)$ is zero if the set of  colors assigned by the configuration $\vec \sigma$ to the vertices in the neighborhood of the vertex $u$ in $G_N$ is a subset of the neighborhood of the color $s$ in the constraint graph $H$ and $Q_{G_N, H, \vec \sigma}(u, s)$ is $-\infty$ otherwise. In other words,   $Q_{G_N, H, \vec \sigma}(u, s)$ is zero  if assigning $u$ to have the color $s$, while leaving the colors  of all other coordinates the same, creates a valid configuration for the $H$-coloring model, and $Q_{G_N, H, \vec \sigma}(u, s)$ is $-\infty$ otherwise.

We  now state the characterization of the MPL estimate for the $H$-coloring model, keeping in mind  Remark \ref{rem-mplabuse}. 

\begin{ppn}
  \label{ppn:Hcoloring}
  Fix $\bm \beta \in \R^q$, and let $\P^N_{\bm \beta}$ be the $H$-coloring measure with
  parameter $\bm \beta$ 
  defined in  \eqref{eq:model_H_II}. Then  for any sample $\vec \sigma \sim \P^N_{\bm \beta}$, any MPL estimate  of $\bm \beta$ must  be a solution of the equation 
 \begin{equation}
  \label{eq:mple_equation_HII} 
   L_{\vec \sigma}^N(\bm b) =\bm 0,
  \end{equation}
 with $\bm 0$ denoting the zero vector of length $q-1$ and $L_{\vec \sigma}^N(\cdot)$ as defined in \eqref{eq:LNb}. 
\end{ppn}

We will show later in Section \ref{sec:pf_gradient_H} that the function $L_{\vec \sigma}^N$  is concave. In fact, $L_{\vec \sigma}^N$  will turn out to be strictly concave with high probability, and hence, \eqref{eq:mple_equation_HII} will have a unique solution with high probability.

We now address  consistency of MPL estimates.  Specifically, 
the following result identifies a sufficient condition under which any sequence of MPL  estimates 
$\{\hat {\bm \beta}^N(\vec \sigma)\}_{N \in \N}$ is $\sqrt N$-consistent for general $H$-coloring models on graphs with bounded average degree.  To state the result, define 
\begin{align}\label{eq:usigma}
\ucons_r(\vec \sigma) & := \sum_{u=1}^{N} \bm 1\{  \vec \sigma_{\cN_{G_N}(u)} \subseteq \cN_{H}(r), 
\end{align}
where recall that $\vec \sigma_{\cN_{G_N}(u)} : = (\sigma_v)_{v \in \cN_{G_N}(u)}$.

\begin{thm}\label{thm:beta_H} Fix $q \geq 2$, $\bm \beta = (\beta_1, \beta_2, \ldots, \beta_{q-1})' \in \R^{q-1}$, and a connected constraint graph $H \ne K_q^+$. For $N \in \N$, let $G_N$  be a graph on $N$ vertices with  $|\Omega_{G_N}(H)| \geq 1$, and
    let the sequence $\{G_N\}_{N \in \N}$ have uniformly bounded average degree, that is,
  suppose \eqref{deg-cond} holds. Moreover, suppose $\vec \sigma \sim \P^N_{\bm \beta}$ is a sample from the $H$-coloring model on $G_N$ as in \eqref{eq:model_H_II} satisfying the following condition: For all $1 \leq r \leq q-1$,   
\begin{align}\label{eq:sigma_neighborhood}
\lim_{\varepsilon \rightarrow 0} \liminf_{N \rightarrow \infty}\P^N_{\bm \beta}\left(\ucons_r(\vec \sigma) > \varepsilon N   \right) = 1. 
\end{align}
Then any sequence of MPL estimates $\{\hat {\bm \beta}^N(\vec \sigma)\}_{N \in \N}$ is $\sqrt N$-consistent for $\bm \beta$, that is,  
$$\lim_{M \rightarrow \infty} \limsup_{N \rightarrow \infty} \P^N_{\bm \beta} (\sqrt N||\hat {\bm \beta}^N(\vec \sigma) - \bm \beta||_2 > M)= 0,$$
where $||\cdot||_2$ denotes the standard Euclidean norm. 
\end{thm}

For ease of referencing, we refer to the sufficient condition  \eqref{eq:sigma_neighborhood} 
as the {\it rainbow  condition}.  The rationale for this terminology is explained in Remark
\ref{remark1}, which also provides more insight into this condition. 

The proof of Theorem \ref{thm:beta_H}, which is   provided  
in Section \ref{sec:pf_beta_H}, involves the following main steps. 
\begin{itemize}
\item
  In the first step, as in the proof of Theorem \ref{thm:beta_estimate}, we show that  the normalized gradient of the log pseudo-likelihood \eqref{eq:mple_equation_HIII} is concentrated around zero at the true model parameter
  using the exchangeable pairs technique (Lemma \ref{lm:2moment_H}). Although the general approach in this step is similar to the hardcore case, the calculations are more cumbersome because of the presence of multiple parameters and the general nature of the constraint graph $H$.
\item
  The second step entails showing that the log pseudo-likelihood is strongly concave, that is, its Hessian matrix is strictly negative definite with high probability. This turns out to be significantly more challenging than the hardcore model, where the special structure of the constraint graph was crucially leveraged to verify the strong concavity of the log pseudo-likelihood. Here, instead, the proof proceeds by first  relating the negative of the Hessian matrix to the covariance matrix of a certain multinomial distribution. Then, using standard facts about  eigenvalues of the covariance matrix of the multinomial distribution, we can show that the minimum eigenvalue of the negative Hessian is positive with high probability whenever \eqref{eq:sigma_neighborhood} holds (see Lemma \ref{lm:condition_H}).
\end{itemize}
This proof also shows that the log-pseudolikelihood is 1-Lipschitz, and hence, the MPL estimate in an $H$-coloring model can be efficiently computed by a simple gradient descent algorithm (as elaborated in Remark \ref{remark:estimate}).

\begin{remark}\label{remark1}
  {\em 
To gain  intuition
into the rainbow condition  \eqref{eq:sigma_neighborhood}, note that 
the  terminology
is motivated  by  the fact that the condition in \eqref{eq:sigma_neighborhood} holds whenever, 
for every color $r \in \{1, 2, \ldots, q-1\}$, 
the expected fraction of vertices with  color $r$ 
is asymptotically positive  (at least  when the  sequence of graphs $\{G_N\}_{N \in \N}$
has uniformly bounded maximum degree).
More precisely, it is shown in   Proposition \ref{ppn:expectation} that for such graph sequences, 
\eqref{eq:sigma_neighborhood}  is implied  by the condition that for
all $\bm \beta'$ in a neighborhood of $\bm \beta$, 
\begin{equation}
  \label{cond-spec}
\liminf_{N \rightarrow \infty}\E^N_{\bm \beta'}[\cons_r(\vec \sigma)] >0,  \forall r \in  \{1, \ldots, q-1\}, 
\end{equation}
where $\cons_r(\vec \sigma)$ 
is the number of vertices with color $r$ in $\vec \sigma$, as defined in \eqref{eq:count_sigma}.
(Note that the color $r=0$ is omitted from the condition in \eqref{cond-spec} because its coefficient $\beta_0=0$ is known.) To  see how the  condition \eqref{cond-spec} is related to   condition \eqref{eq:sigma_neighborhood}, 
note that for fixed $r \in \{1, \ldots,  q-1\}$, by invoking   \eqref{eq:usigma} and \eqref{eq:Q_sigma}
  $\ucons_r(\vec \sigma)$  can be rewritten as follows: 
      \begin{align*} 
\ucons_r(\vec \sigma)  =\sum_{u=1}^{N} \bm 1\{  \vec \sigma_{\cN_{G_N}(u)} \subseteq \cN_{H}(r) \} &
= \sum_{u=1}^{N} \bm 1\{Q_{G_N, H, \vec \sigma}(u, s) = 0\}  \nonumber \\ 
& = |\cU^N_r(\vec\sigma)|, 
\end{align*}
where
\begin{align}\label{eq:unconstrained_r}
  \cU^N_r(\vec\sigma) & :=\{ u \in V(G_N):  \sigma_v \in \cN_H(r). 
  \forall  v \in \cN_{G_N}(u)\}
\end{align}
Note that $\cU^N_r(\vec\sigma)$ 
is the set of vertices in $G_N$ for which all neighbors have colors in the $H$-neighborhood of $r$.
We refer to this set as the collection of $r$-{\it unconstrained} vertices, because, for every $u \in \cU^N_r(\vec \sigma)$, the (possibly new) spin configuration obtained by assigning the color $r$ to the vertex $u$   
  and leaving the colors of all other vertices unchanged, still forms a valid $H$-coloring of $G_N$. With this terminology,  the  condition in \eqref{eq:sigma_neighborhood} implies that for every $1 \leq r \leq q-1$, the number of $r$-{\it unconstrained} vertices is $\Theta(N)$  with probability going to 1. However, note that  
\begin{align}\label{eq:Ur_condition}
\ucons_r(\vec \sigma) \geq \sum_{u=1}^N \bm 1\{\sigma_u=r\} = \cons_r(\vec \sigma). 
\end{align} 
Thus,  a useful sufficient condition for \eqref{eq:sigma_neighborhood}  to hold  
is that for every $1 \leq r \leq q-1$, the number of vertices in $G_N$ that  are assigned color $r$ by $\vec \sigma$ is $\Theta(N)$,  with high probability under $\P_{\bm \beta}^N$.
Using the fact that $\cons_r(\vec\sigma)$ appears as a sufficient statistic in  \eqref{eq:model_H}, one
can apply Markov's inequality to relate this latter condition to the estimate
\eqref{cond-spec}.
For a rigorous justification,    see the proof of  Proposition \ref{ppn:expectation} in 
Section \ref{sec:pf_expectation}.}
\end{remark}

Theorem \ref{thm:beta_H} can be easily applied to the various classes of  $H$-coloring models described above. One important class of examples are constraint graphs where there is at least one  vertex with no constraints. More formally, a  constraint graph $H$ is said to have an {\it unconstrained} vertex $s \in H$, if $(s, t) \in E(H)$, for all $t \in V(H)$. This includes, for example, the multi-state hardcore model and the Widom-Rowlinson model, in both of which the vertex representing the color $0$ is unconstrained. We show in the corollary below, that in models where $H$ has at least one unconstrained vertex, the MPL estimate is $\sqrt N$-consistent whenever $\{G_N\}_{N \in \N}$ has uniformly bounded average degree.

\begin{cor}\label{cor:unconstrained_H} Fix $q \geq 2$, $\bm \beta = (\beta_1, \beta_2, \ldots, \beta_{q-1})' \in \R^{q-1}$, and a connected constraint graph $H \ne K_q^+$ that has an unconstrained vertex. Moreover, suppose for
  each $N \in \N$, $G_N$ is a graph on  $N$ vertices and the sequence $\{G_N\}_{N \in \N}$ satisfies the uniformly bounded average  degree condition
  \eqref{deg-cond}.  
  Then, given $\vec \sigma \sim \P^N_{\bm \beta}$ from the $H$-coloring model \eqref{eq:model_H_II}, 
  any sequence of MPL estimates $\{\hat {\bm \beta}^N(\vec \sigma)\}_{N \in \N}$ is $\sqrt N$-consistent for $\bm \beta$. 
\end{cor}

The proof of this corollary is given in Section \ref{sec:pfunconstrained_H}. It entails showing $\liminf_{N \rightarrow \infty}\E^N_{\bm b}[\cons_r(\vec \sigma)] >0$, for all $1 \leq r \leq q-1$ and all $\bm b \in \R^{q-1}$ (recall the discussion in Remark \ref{remark1}).  In  particular,
Corollary \ref{cor:unconstrained_H} 
implies the $\sqrt N$-consistency of the multi-state hardcore model and the Widom-Rowlinson model whenever $|E(G_N)|=O(N)$.

However, Corollary \ref{cor:unconstrained_H}  does not apply to the $q$-coloring model because its constraint graph has no vertex that is unconstrained. In fact, we show in Section \ref{subs-eg1} that   
in this case, bounded average degree is not always enough to ensure consistent estimation. In particular, 
 the rainbow condition is violated and  the result of Theorem \ref{thm:beta_H} do not hold. However, interestingly, as the next corollary shows, consistent estimation is still possible for 
   $q$-coloring models if we strengthen the condition on the graph sequence $\{G_N\}_{N \in  \N}$.
   Specifically,  for $v \in V(G_N)$, let  $d_{2,  v}(G_N)$ be the number of vertices that are at distance either 1 or 2 from $v$ in $G_N$.   Then we have the following result.

\begin{cor}\label{cor:coloring} Fix $q \geq 2$ and $\bm \beta = (\beta_1, \beta_2, \ldots, \beta_{q-1})' \in \R^{q-1}$.  Suppose for
    each $N \in \N$, $G_N$ is a graph on  $N$ vertices with $|\Omega_{G_N}(K_q)| \geq 1$, and the sequence $\{G_N\}_{N \in \N}$  has uniformly bounded average 2-neighborhood, that is, 
\begin{equation}
  \label{deg2-cond}
 \sup_{N \geq 1}  \frac{1}{N} \sum_{v=1}^N d_{2,  v}(G_N) < \infty. 
  \end{equation} 
Then given a sample $\vec \sigma \sim \P^N_{\bm \beta}$ from the $q$-coloring model (that is, the
$H$-coloring model with $H=K_q$), any sequence of MPL estimates $\{\hat {\bm \beta}^N(\vec \sigma)\}_{N \in \N}$ is $\sqrt N$-consistent for $\bm \beta$. 
\end{cor} 

Note that this result, in particular, implies $\sqrt N$-consistency of the MPL estimates for $q$-coloring models on sequences of graphs with uniformly bounded maximum degree and the sparse Erd\H os-R\'enyi random graph sequence $\{{\mathcal G}(N, c/N)\}_{N \in \N}$, for some fixed $c > 0$, since these have   bounded average 2-neighborhoods.

In view of  Corollary \ref{cor:coloring} it is natural to  ask  if 
  consistent estimation may be possible for all $H$-coloring models, even  when
  the  rainbow  condition is violated, as long as a more stringent condition on the graph sequence,
such as   \eqref{deg2-cond}, 
  is imposed. 
  In Section  \ref{subs-eg1} we answer this question in  the negative.
  Specifically,  we construct a connected constraint graph $H$
          and a sequence of uniformly bounded maximum degree graphs $\{G_N\}_{N \in \N}$ where one of the colors can   only appear $O(1)$ times in any valid $H$-coloring of $G_N$, and 
          the rainbow  condition \eqref{eq:sigma_neighborhood} is violated.
          For  this example, we also show that it is impossible to consistently estimate parameters, thus illustrating the tightness of 
        the rainbow condition \eqref{eq:sigma_neighborhood} in obtaining $\sqrt N$-consistent
        estimates for general models  even under the stronger condition
          that  the  graph  sequences  have  uniformly bounded
          maximum degree. 
        Furthermore,   to address the tightness of the conditions in Corollary \ref{cor:coloring},         
          in Section \ref{subs-eg2} we construct a $q$-coloring model on a sequence of graphs with bounded average degree where consistent estimation is also impossible, thus demonstrating the need for a stronger condition  on  the graph  sequence, like the one  we have imposed  in \eqref{deg2-cond} on the  $2$-neighborhood.

Thus, the proof of Corollary \ref{cor:coloring}, which is given in Section \ref{sec:pfcorollary_II},
  makes crucial use of both the stucture of the $q$-coloring model, and the $2$-neighborhood
  condition \eqref{deg2-cond}. 
  The main ingredient of the proof is the following combinatorial observation: Any graph with bounded average 2-neighborhood has a subset of vertices of size $\Theta(N)$ whose neighborhoods in $G_N$ are mutually disjoint (see Lemma \ref{lm:2neighborhood}). Once such a set is guaranteed to exist, 
condition \eqref{eq:sigma_neighborhood} can be easily verified using the observation that when $H=K_q$, the set of $r$-unconstrained vertices $\cU_r^N(\vec \sigma)$, as defined in \eqref{eq:unconstrained_r}, is equal to $\{ u \in V(G_N): r \notin \cN_{G_N(u)}\}$. The combinatorial observation is proved by a probabilistic method argument, which mimics the proof of a celebrated theorem  of Turan \cite[Page 95]{independent_set_book}. One of the consequences of Turan's theorem is that any graph with bounded average degree (equivalently, 1-neighborhood) has an independent set of size $\Theta(N)$. Drawing parallels from this, here we show that bounded average 2-neighborhood implies that the graph has a linear size subset of vertices that are distance at least 2 away from each other, that is, they have mutually disjoint neighborhoods.

\subsection{Organization} The rest of the paper is organized as follows: In Section \ref{sec:pfbeta_estimate} we prove Theorem \ref{thm:beta_estimate}. Proposition \ref{ppn:Hcoloring} is proved in Section \ref{sec:MPLpf}. The proof of Theorem \ref{thm:beta_H} is given in Section \ref{sec:pf_beta_H}. Another sufficient condition for \eqref{eq:sigma_neighborhood} and the proof of Corollary  \ref{cor:unconstrained_H} are given in Section \ref{sec:pfcorollary}. The proof of Corollary \ref{cor:coloring} is given in Section \ref{sec:pfcorollary_II}. Examples where consistent estimation is impossible are discussed in Section \ref{sec:examples}.

\section{Proof of Theorem \ref{thm:beta_estimate}}
\label{sec:pfbeta_estimate}

Throughout  this section we will  suppress the dependence on $\vec \sigma$ and denote by $\hat \beta^N:=\estimate $ and $\hat \lambda^N :=\hat \lambda^N (\vec \sigma)$ the MPL estimate of $\beta$ and $\lambda$, respectively, as defined in \eqref{eq:beta_estimate} and \eqref{eq:lambda_estimate}, respectively.  We also let
  $\E_\beta^N$  denote expectation with respect to $\P_\beta^N$. 
The first step in the proof of Theorem \ref{thm:beta_estimate} is to show concentration of the derivative of the log
pseudo-likelihood function about zero at the true parameter of the model. To this end, in Lemma \ref{lm:2moment},   we establish a second moment bound for $L^N_{\vec \sigma}(\cdot)$.

\begin{lem}\label{lm:2moment} Given $\vec \sigma \sim \P^N_\beta$, let $L^N_{\vec \sigma}(\cdot)$ be as defined in \eqref{eq:mple_solution}. Then, under the assumptions of Theorem \ref{thm:beta_estimate}, for any $\beta \in \R$, 
$$\E_\beta^N[L_{\vec \sigma}^N(\beta)^2] \lesssim_\beta \frac{1}{N}.$$
\end{lem}
\begin{proof} Fix $N \in \N$ and write $L_{\vec \sigma} =  L_{\vec \sigma}^N$  and $\E_{\bm \beta} = \E^N_{\bm \beta}$.    
  For any $1 \leq J \leq N$ and $\vec \tau \in \{0, 1\}^N$, let 
\begin{equation}
    \label{tauL}
    \vec \tau_{(J)}=(\tau_1, \tau_2, \ldots, \tau_{J-1}, 1- \tau_J, \tau_{J+1}, \ldots, \tau_N)
\end{equation}
be the configuration obtained by switching the state of the $J$-th vertex in $\vec \tau$. 
Next, for $\vec \tau, \vec \tau^* \in \{0, 1\}^N$, define $\newF(\vec \tau, \vec \tau^*):=\frac{1}{N} \sum_{i=1}^N(\tau_i-\tau_i^*)$.  
  Now choose a coordinate $I$ uniformly at random from $\{1, \ldots, N\}$ and replace the $I$-th coordinate of $\vec \sigma$ by a sample drawn from the conditional distribution of $\sigma_I$ given $(\sigma_v)_{v \ne I}$. Denote the resulting vector by $\vec \sigma'$.  Note that
  \begin{equation}
    \label{F-rel}
    \newF(\vec \sigma, \vec \sigma')=\sigma_I-\sigma_I'. 
  \end{equation} 
  Along with  \eqref{cond-prob} and \eqref{eq:mple_solution}, and the convention  $\log 0  = -\infty$,  this implies 
  \begin{align}\label{eq:mple_solution_III}
f(\vec \sigma) & :=\E_{\beta}[\newF(\vec \sigma, \vec \sigma')|\vec \sigma] \nonumber \\ 
&=\frac{1}{N}  \sum_{v=1}^N (\sigma_v-\E_{\beta}[\sigma_v|(\sigma_u)_{u \ne v}]) \nonumber \\
&= \frac{1}{N} \sum_{u=1}^N \left\{ \sigma_u -  \frac{ e^{\beta + \sum_{v \ne u}  a_{uv}(G_N) \log  (1- \sigma_v)} }{ e^{\beta + \sum_{v \ne u}  a_{uv}(G_N) \log  (1- \sigma_v)}+ 1 } \right\}  \nonumber \\ 
&= L_{\vec \sigma}(\beta). 
  \end{align}
   Multiplying both sides of the last display by
 $f(\vec \sigma)$ and taking expectations, we see that
 \[ \E_{\beta}[f(\vec \sigma)^2]= \E_{\beta}[f(\vec \sigma) \newF(\vec \sigma, \vec \sigma')]. \]
 Now, since $(\vec \sigma, \vec \sigma')$ is an exchangeable pair,  and $\newF$ is anti-symmetric, wee have 
\begin{align*} 
\E_{\beta}[f(\vec \sigma) \newF(\vec \sigma, \vec \sigma') ] & = \E_{\beta}[f(\vec \sigma') \newF(\vec \sigma', \vec \sigma)] \nonumber \\ 
& =-\E_{\beta}[f(\vec \sigma') \newF(\vec \sigma, \vec \sigma')].
\end{align*} 
The last two displays, when combined with \eqref{eq:mple_solution_III},   imply  
\begin{align}\label{eq:2moment_pf_I}
\E_{\beta}[L_{\vec \sigma}(\beta)^2] & =\E_{\beta}[f(\vec \sigma)^2] \nonumber \\ 
& =\frac{\E_{\beta}[(f(\vec \sigma)-f(\vec \sigma'))\newF(\vec \sigma, \vec \sigma')]}{2}.
\end{align}   
Now,  define $p_J(\vec \tau):=\P^N_{\beta}(\sigma_J' =1-\tau_J|\vec \sigma=\tau, I=J),$ for $J \in \{1, \ldots, N\}$, recall the definition of $\sigma_{(J)}$  from \eqref{tauL} and use the fact that $\newF(\sigma, \sigma_{(J)}) = 1 - 2 \sigma_J$,
to obtain 
\begin{align}\label{eq:2moment_pf_II}
 \E_{\beta}  [(f(\vec \sigma)-f(\vec \sigma'))\newF(\vec \sigma, \vec \sigma')|\vec \sigma] & =\frac{1}{N}  \sum_{J=1}^N (f(\vec \sigma)-f(\vec \sigma_{(J)})) \newF(\vec \sigma, \vec \sigma_{(J)}) p_J(\vec \sigma) \nonumber \\ 
&= \frac{1}{N} \sum_{J=1}^N (f(\vec \sigma)-f(\vec \sigma_{(J)})) (1 - 2 \sigma_J) p_J(\vec \sigma). 
  \end{align} 
Since \eqref{eq:mple_solution_III} and \eqref{LN-rewrite}  imply  
 $f(\vec \sigma)=L_{\vec \sigma}(\beta)=\frac{1}{N}  \left( \cons(\vec \sigma) -  e^{\beta} \ucons(\vec \sigma)\right)/\left(e^{\beta}+1\right)$  and  the definition of $c^N$ in \eqref{def-aN} implies $\cons(\vec \sigma)-\cons(\vec \sigma_{(J)})= 1 - 2 \sigma_J$, we have 
\begin{align}\label{eq:2moment_pf_III}
 f(\vec \sigma)-f(\vec \sigma_{(J)}) & =\frac{ (1 - 2 \sigma_J)- e^{\beta} (\ucons(\vec \sigma) - \ucons(\vec \sigma_{(J)})) }{N(e^{\beta}+1)} . 
\end{align}

Combining \eqref{eq:2moment_pf_I}, \eqref{eq:2moment_pf_II},  \eqref{eq:2moment_pf_III}, and the fact that $1 - 2 \sigma_J \in \{-1,1\}$, which implies $(1 - 2 \sigma_J)^2 = 1$, 
it follows that 
\begin{align}\label{eq:T12}
\E_{\beta}[L_{\vec \sigma}^2] & = \tfrac{1}{2} (T_1 - T_2) , 
\end{align}
where for $i = 1, 2,$ $T_i = T_i^N (\vec \sigma)$ are given by 
\begin{align*}
T_1 & := \frac{\E_{\beta} \left[ \sum_{J=1}^N  p_J(\vec \sigma) \right]}{ N^2 (e^{\beta}+1)}, \nonumber \\ 
T_2 & := \frac{\E_{\beta} \left[ e^{\beta} \sum_{J=1}^N  (\ucons(\vec \sigma) - \ucons(\vec \sigma_{(J)})) (1 - 2\sigma_J) p_J(\vec \sigma) \right]}{ N^2 (e^{\beta}+1)}. 
\end{align*}
Note that $T_1 \lesssim_\beta \frac{1}{N}$, since $p_J(\vec \sigma) \leq 1$.
On the other hand, by the definition of $\ucons$ in \eqref{def-uN}, it follows that 
\[ |\ucons(\vec \sigma) - \ucons(\vec \sigma_{(J)})| \leq d_J(G_N)+1,\]
where  recall  $d_J(G_N)$ is the degree of the vertex $J$ in $G_N$.
Hence, 
$$T_2 \lesssim_\beta \frac{1}{N^2} \sum_{J=1}^N d_J(G_N)  \lesssim\frac{1}{N},$$
where the last step uses the assumption \eqref{deg-cond}. Due  to \eqref{eq:T12}, this implies $\E_{\beta}[f(\vec \sigma)^2] \lesssim_\beta \frac{1}{N}$, which completes the proof of the lemma. 
\end{proof}

We now proceed to show that the log pseudo-likelihood for the hardcore model
is strictly concave  whenever $\{G_N\}_{N \geq 1}$ has uniformly bounded average degree. To this end,
note from \eqref{LN-rewrite} that
$$
\frac{\mathrm d}{\mathrm d\beta}  L_{\vec \sigma}^N(\beta) = - \frac{e^\beta}{e^\beta + 1} \cdot \frac{\cons(\vec \sigma) + \ucons(\vec \sigma) }{N},
$$ 
with  $\cons(\vec \sigma)$ and $\ucons(\vec \sigma)$ as defined in \eqref{def-aN} and \eqref{def-uN}, respectively. Therefore, to show the strict concavity of the log-pseudolikelihood function, it suffices to show that $\cons(\vec \sigma)$ is $\Theta(N)$,  with high probability. As we show below,  this, in turn, follows from the fact that $\frac{1}{N}\frac{\mathrm d}{\mathrm d b} F_{G_N}(\cdot)$, the derivative 
of the normalized log partition function,  is strictly positive in the limit.

\begin{lem}\label{lm:lp}  Suppose $\{G_N\}_{N \geq 1}$ is a sequence of graphs with uniformly bounded average degree, that is, satisfying
  \eqref{deg-cond}. 
  Then, for all $b \in \R$, 
\begin{align}\label{eq:lp_derivative}
\liminf_{n\rightarrow \infty} \frac{1}{N} F_{G_N}'(b) > 0.
\end{align} 
As a consequence, for $\vec \sigma \sim \P^N_\beta$, 
\begin{align}\label{eq:asigma_epsilon}
\limsup_{\varepsilon \rightarrow 0} \limsup_{N\rightarrow\infty}\P^N_{\beta}(\cons(\vec \sigma) \leq \varepsilon N)=0. 
\end{align}
\end{lem}
\begin{proof}
  Denote the collection of all independent sets of $G_N$ by $\cI(G_N)$. Then, fixing $u \in V(G_N)$, we partition  $\cI(G_N)$ as follows:  Let $\cI_u^+(G_N)$ be the collection of independent sets of $G_N$ containing the vertex $u$, and $\cI_u^-(G_N)$ the collection of independent sets of $G_N$ not containing the vertex $u$. Now, define the map $g: \cI_u^-(G_N) \rightarrow \cI_u^+(G_N)$, which takes $A \in \cI_u^-(G_N)$ to $(A \backslash \cN_{G_N}(u)) \bigcup \{u\} \in \cI_u^+(G_N)$.  Note that given $A' \in \cI_u^+(G_N)$, 
\begin{align}\label{eq:g_I} 
|g^{-1}(A')| & :=|\{A \in  \cI_u^-(G_N): g(A)=A'\}|  \leq  2^{|\cN_{G_N}(u)|} \leq 2^{d_u(G_N)}. 
\end{align} 
Moreover, for $A \in g^{-1}(A')$, $|A'| - d_u(G_N) \leq |A'| - 1 \leq  |A| \leq |A'| + {d_u(G_N)}$, which gives $e^{b |A|} \leq e^{|b| {d_u(G_N)} } e^{b |A'|}$. Together with \eqref{eq:g_I} this implies  
\begin{align}\label{eq:Z_bound}
\sum_{A \in \cI^-_u(G_N)} e^{b |A|} & \leq \sum_{A' \in \cI^+_u(G_N)} \sum_{A \in g^{-1}(A')} e^{b |A|}  \nonumber \\  
& \leq  e^{|b| {d_u(G_N)} } 2^{d_u(G_N)} \sum_{A' \in \cI^+_u(G_N)} e^{b |A'|}. 
\end{align}
Hence, it follows that 
\begin{align}\label{eq:sigmau}
\P^N_{b}(\sigma_u=1) = \frac{\sum_{A \in \cI^+_u(G_N)}e^{b |A|}}{\sum_{A \in \cI(G_N)} e^{b |A|}}  &= \frac{\sum_{A \in \cI^+_u(G_N)}e^{b |A|}}{\sum_{A \in \cI^+_u(G_N)} e^{b |A|} + \sum_{A \in \cI^-_u(G_N)} e^{b |A|}} \nonumber \\ 
&\geq \frac{1}{1+  e^{|b| {d_u(G_N)} } 2^{d_u(G_N)}}, 
\end{align}
where the last step  uses \eqref{eq:Z_bound}.
Now, noting from \eqref{eq:FGn}, \eqref{eq:pmf_beta} and \eqref{def-aN} that $F_{G_N}'(b) = \E_{b}^N[\cons(\vec \sigma)]$,
and using \eqref{eq:sigmau} yields, for any $M > 0$,  
\begin{align}\label{eq:FGN_bound}
\frac{1}{N}F_{G_N}'(b)  = \frac{1}{N}\E_{b}^N[\cons(\vec \sigma)] & = \frac{1}{N} \sum_{u=1}^N \P^N_{b}(\sigma_u=1) \nonumber \\  
& \geq \frac{1}{N} \sum_{u=1}^N \frac{1}{1+  e^{|b| {d_u(G_N)} } 2^{d_u(G_N)}} \nonumber \\ 
& \geq \frac{1}{1+  e^{|b| M } 2^{M}} \frac{|\{u: d_u (G_N)\leq M \}|}{N}  . 
\end{align}
We will now show that if $\frac{1}{N}\sum_{u=1}^N d_u (G_N) = O(1)$, then there exists $M_0 \geq  1$  such that the right-hand side in \eqref{eq:FGN_bound} is $\Omega(1)$.

  \begin{lem}\label{lm:M} Suppose $\{G_N\}_{N \in  \N}$ is a sequence of graphs with $\sup_{N \geq 1} \frac{1}{N} \sum_{u=1}^N d_u (G_N) < \infty$. Then there exists $0 < M_0 <  \infty$ such that for every  $N \in \N$, 
    \[ \frac{|\{u: d_u(G_N) \leq M_0\}|}{N} \geq \frac{1}{2}. \]
\end{lem}

  \begin{proof} Let $X_{G_N}$ be the discrete random variable that takes the  value $d_v(G_N)$ with probability $\frac{1}{N}$, 
    for $v \in \{1,\ldots, N\}$. 
    Then choosing $M_0:= 2 \sup_{N \in \N} \frac{1}{N} \sum_{u=1}^N d_u (G_N) < \infty$ and using Markov's inequality gives, 
\begin{align*} 
  \frac{|\{u: d_u (G_N) > M_0\}|}{N}  = \P(X_{G_N} > M_0)  & \leq \frac{\E[X_{G_N}]}{M_0} \nonumber \\ 
& = \frac{1}{M_0} \cdot \frac{1}{N}\sum_{u=1}^N d_u(G_N) \nonumber \\ 
& \leq \frac{1}{2}.
\end{align*} 
This completes the proof of the lemma. 
\end{proof}

The above lemma can now be used to complete the proof of  \eqref{eq:lp_derivative}. First, by Lemma \ref{lm:M} choose $M_0 \in (0,\infty)$ such that $\liminf_{N \rightarrow \infty} \frac{|\{u: d_u(G_N) \leq M_0\}|}{N} \geq \frac{1}{2}$. Then choosing $M=M_0$  in \eqref{eq:FGN_bound} and taking  the limit as $N \rightarrow \infty$ yields 
\begin{align*}
\liminf_{N \rightarrow \infty} \frac{1}{N}F_{G_N}'(b) & \geq \frac{1}{1+  e^{|b| M_0 } 2^{M_0}} \liminf_{N \rightarrow \infty} \frac{|\{u: d_u(G_N) \leq M_0 \}|}{N} \nonumber \\  
 & > 0. 
\end{align*}
This proves \eqref{eq:lp_derivative}.

To complete the proof of Lemma \ref{lm:lp},  it only remains to  establish \eqref{eq:asigma_epsilon}.
To this end, fix $\varepsilon, \delta>0$ and note by Markov's inequality,   
\begin{align*}
\P^N_{\beta}(\cons(\vec \sigma)<\varepsilon N)  = \P^N_{\beta}(e^{- \delta a(\vec \sigma)}>e^{- \delta \varepsilon N})  & \le e^{ \delta \varepsilon N} \E_{\beta}^N [ e^{-\delta \cons(\vec \sigma)}]  \nonumber \\ 
& \le e^{ \delta \varepsilon N+F_{G_N}(\beta-\delta)-F_{G_N}(\beta)}. 
\end{align*}
On taking logarithms of both sides above, we obtain
\begin{align*}
\log \P^N_{\beta}(\cons(\vec \sigma)<\varepsilon N) & \le \varepsilon\delta N-\int_{\beta-\delta}^{\beta}F_{G_N}'(t)\mathrm dt \nonumber \\   
& \le \varepsilon \delta N -F_{G_N}'(\beta-\delta)\delta,
\end{align*}
where the last step uses the monotonicity of the function $F_{G_N}'(\cdot)$, which holds because  $F_{G_N}''(b)=\Var_{b}(\cons(\vec \sigma)) \geq 0$, for all $b \in \R$. Dividing both sides of the last display by $N$ and taking limits, first as $N\rightarrow \infty$ followed by $\varepsilon \rightarrow 0$, we obtain  
\begin{align*} 
\limsup_{\varepsilon \rightarrow 0}  \limsup_{N\rightarrow\infty}\frac{1}{N}\log \P^N_{\beta}(\cons(\vec \sigma)<\varepsilon N) & \le   -\delta \liminf_{N\rightarrow\infty}\frac{1}{N}{F_{G_N}'(\beta-\delta)} \nonumber \\ 
& < 0, 
\end{align*}
where the last inequality uses \eqref{eq:lp_derivative}. This proves \eqref{eq:asigma_epsilon}.  
\end{proof}

The above two lemmas can be used to complete the proof of Theorem \ref{thm:beta_estimate}. 

\begin{proof}[Proof  of Theorem \ref{thm:beta_estimate}.]
  For notational simplicity, we will denote $L_{\vec \sigma} := L^N_{\vec \sigma}$ and
  $M_{\vec \sigma} := M^N_{\vec \sigma}$. 
To begin with,  recalling \eqref{LN-rewrite}, define 
\begin{align}\label{eq:L_beta_I}
\bar L_{\vec \sigma}(\beta) :=(e^{\beta}+1)L_{\vec \sigma}(\beta)=\frac{\cons(\vec \sigma) -e^{\beta} \ucons(\vec \sigma)}{N}.
\end{align} 
Next, fix $\delta>0$. Then Markov's inequality and Lemma \ref{lm:2moment} show that  for $K=K(\delta, \beta) > 0$ large enough (not depending on $N$), 
\begin{eqnarray}\label{chebyshev}
\P^N_{\beta}\left(|\bar L_{\vec \sigma}(\beta)|> \frac{K}{\sqrt N} \right)  \lesssim_{\beta} \frac{1}{K^2} \leq \frac{\delta}{2}. 
\end{eqnarray}
  Moreover, choose $\varepsilon:=\varepsilon(\delta, \beta) > 0$ such that $\P^N_\beta(\cons(\vec \sigma) \leq \varepsilon N) \leq \tfrac{\delta}{2}$, for all $N$ large enough. Note this  is possible by \eqref{eq:asigma_epsilon} of Lemma \ref{lm:lp}. Define 
  $$T_{N}:=\left\{\vec \sigma \in \{0, 1\}^N: |\bar L_{\vec \sigma}(\beta)| \le K/\sqrt N \text{ and } \cons(\vec \sigma) >  \varepsilon N \right\}.$$ Then, by \eqref{chebyshev} and the subsequent observation, $\P^N_{\beta}(T_N)\ge 1-\delta$. Now, for $\vec \sigma \in T_N$,
 using \eqref{eq:L_beta_I} for  the first  two equalities below, it follows that 
\begin{align}\label{eq:Lderivative}
  - \frac{\mathrm d}{\mathrm d \beta}\bar L_{\vec \sigma}(\beta)   =\frac{e^\beta \ucons(\vec \sigma)}{N} =\frac{\cons(\vec \sigma)}{N} - \bar L_{\vec \sigma}(\beta)  \geq \varepsilon - \frac{K}{\sqrt N}  \geq \frac{\varepsilon}{2},
\end{align}
for all $N$ large enough. Therefore, for $\vec \sigma \in T_N$, recalling $\bar L_{\vec \sigma}(\hat \beta^N) = 0$ by the definition of $\hat{\beta}^N$  and \eqref{eq:L_beta_I}, we have for all $N$ large enough, 
\begin{align*}
\frac{K}{\sqrt N} \geq ~|\bar L_{\sigma}(\beta)|   =|\bar L_{\sigma}(\beta)-\bar L_{\sigma}(\hat \beta^N)| & = - \int_{\min\{\beta,  \hat \beta^N\}}^{\max\{\beta, \hat \beta^N\} }\frac{\mathrm d}{\mathrm d \beta} \bar L_{\vec \sigma}(t)\mathrm d t \nonumber \\ 
& \gtrsim \varepsilon  |\hat \beta^N - \beta |,
\end{align*}
where the last inequality uses \eqref{eq:Lderivative}. This implies that with probability at least $1-\delta$, $|\hat \beta^N-\beta|=O_{\delta, \beta}(1/\sqrt N)$. Since $\delta > 0$ is  arbitrary,  this completes the proof of Theorem \ref{thm:beta_estimate}.
\end{proof}

\section{Proof of Proposition \ref{ppn:Hcoloring}}
\label{sec:MPLpf}

We fix a constraint graph $H$, set $q  := |V(H)|$ and fix a finite graph $G_N$ on $N$ vertices. 
We begin by obtaining a convenient expression for the conditional probability
$\P_{\bm \beta}^N (\sigma_u = r|(\sigma_v)_{v \neq u})$ for any $u \in G_N$.

  \begin{lem}
    \label{condprob-Hmodel}
    For any $\bm{\beta} = (\beta_0, \beta_1, \ldots, \beta_{q-1}) \in  \R^q$ with $\beta_0 := 0$
    and $r = 0, 1, \ldots, q-1$, 
    \begin{align}\label{eq:sigma_r}
      \P^N_{\bm \beta}  (\sigma_u=r |(\sigma_v)_{v \ne u})
      &=\frac{e^{\beta_r  + Q_{G_N, H, \vec \sigma}(u, r)}}{\sum_{s=0}^{q-1} e^{\beta_s  + Q_{G_N, H, \vec \sigma}(u, s)}}.  
    \end{align}
    where  $Q_{G_N, H, \vec \sigma}$ is as defined in \eqref{eq:Q_sigma0}. 
  \end{lem}
  \begin{proof}
    Fix $0 \leq r \leq q-1$. 
    Suppose  $\vec \sigma \sim \P^N_{\bm \beta}$ is such that $\sigma_u=r$ for some $u \in V(G_N)$. Then, in view of the expression for $\P^N_{\bm \beta}$ in \eqref{eq:model_H_II}, recalling from \eqref{chi} that
    $\chi_{\vec \sigma, s}(w) = \bm 1\{\sigma_w =s \}$,  for $w \in V(G_N)$ and $0 \leq s \leq q-1$,
we have 
\begin{align*} 
\sum_{s=1}^{q-1}  \beta_s \cons_s(\vec \sigma) & = \sum_{v \in V(G_N)} \sum_{s=1}^{q-1}  \beta_s \chi_{\vec \sigma, s}(v)    = \sum_{v \in V(G_N) \backslash \{ u \}} \sum_{s=1}^{q-1}  \beta_s \chi_{\vec \sigma, s} (v) + \beta_r .
\end{align*} 
Moreover, when $\sigma_u=r$, by  \eqref{eq:model_H}, it follows  that  
\begin{align*}
\log \mathcal \newC_{G_N, H}(\vec \sigma) =& \sum_{(s, t) \in E(\overline{H})} \sum_{1\leq u \ne v \leq N}  a_{uv}(G_N) \log (1-\chi_{\vec \sigma, u}(s) \chi_{\vec \sigma, v}(t)) \nonumber \\
& =  \frac{1}{2} \sum_{v \in V(G_N) \backslash\{u\}}  \sum_{t=0}^{q-1}  a_{rt}(\overline H) a_{uv}(G_N) \log (1- \chi_{\vec \sigma, v}(t))  + S ((\sigma_v)_{v \ne u}) , 
\end{align*}
where $S$ is a  function that depends only on the coordinates $(\sigma_v)_{v \ne u}$.
This implies 
\begin{align*}
\P^N_{\bm \beta}(\sigma_u=r |(\sigma_v)_{v \ne u}) & =  \frac{\P^N_{\bm \beta} (\sigma_u = r, (\sigma_v)_{v \in u})}{\sum_{s=0}^{q-1} \P^N_{\bm \beta} (\sigma_u = s, (\sigma_v)_{v \in u})}  =\frac{e^{\beta_r  + Q_{G_N, H, \vec \sigma}(u, r)}}{\sum_{s=0}^{q-1} e^{\beta_s  + Q_{G_N, H, \vec \sigma}(u, s)}}, 
\end{align*} 
recalling the definition of $Q_{G_N, H, \vec{\sigma}}(\cdot, \cdot)$ from \eqref{eq:Q_sigma}.  This proves
\eqref{eq:sigma_r}. 
    \end{proof}

  The proof of Proposition \ref{ppn:Hcoloring} is  now an easy consequence of the lemma.
  Indeed, it  only remains to show that
  \begin{equation}
   \label{eq:mple_equation_HIII}   L_{\vec \sigma}^N(b) = \frac{1}{N}  \nabla M_{\vec \sigma}^N (b),
    \end{equation}
  where 
  \begin{equation}
    \label{def:logpl}
    M_{\vec \sigma}^N (b) := \log \prod_{u=1}^N \P_{\bm b}^N (\sigma_u|(\sigma_v)_{v \neq  u}).
  \end{equation}
  
\begin{proof}[Proof of Proposition \ref{ppn:Hcoloring}]
  In view of Definition \ref{def:MPL} and   Remark \ref{rem-mplabuse}
the MPL estimate of the parameter $\bm \beta$ must be a critical point 
of the log pseudo-likelihood
function $M^N_{\vec \sigma}(\bm b)$.  From \eqref{def:logpl} and  \eqref{eq:sigma_r},  it follows that 
\begin{align}\label{eq:mple_function}
  M^N_{\vec \sigma}(\bm b)
 =  \sum_{u=1}^N \Bigg\{ b_{\sigma_u}+ Q_{G_N, H, \vec \sigma}(u, r)  -  \log\left( \sum_{s=0}^{q-1} e^{b_s + Q_{G_N, H, \vec \sigma}(u, s)} \right) \Bigg\},  
\end{align}
where we recall $b_0=0$. To   show \eqref{eq:mple_equation_HIII} note that 
 for $1 \leq r \leq q-1$, using \eqref{eq:mple_function} and  the fact that  \eqref{eq:Q_sigma} implies 
$Q_{G_N, H, \vec \sigma}(u, r)  \in \{0,-\infty\}$,  we have  
 \begin{align*}
   \frac{1}{N}\frac{\partial}{\partial b_r}M^N_{\vec \sigma}(\bm b)    &=\frac{1}{N}\sum_{u=1}^N \left\{ \bm 1\{\sigma_u =r \} -  \frac{  e^{b_r + Q_{G_N, H, \vec \sigma}(u, r)} }{ \sum_{s=0}^{q-1} e^{b_s + Q_{G_N, H, \vec \sigma}(u, s)}} \right\}
  \notag \\
  &= \frac{1}{N} \cons_r(\vec \sigma)   - \frac{1}{N} \sum_{u=1}^N \left\{ \frac{  e^{b_r} \bm 1\{Q_{G_N, H, \vec \sigma}(u, r)=0\} }{  \sum_{s=0}^{q-1} e^{b_s} \bm 1\{Q_{G_N, H, \vec \sigma}(u, s)=0\}} \right\}, 
 \end{align*}
 which coincides with $L^{(r)}_{\vec \sigma}(\bm b)$  by \eqref{eq:mple_equation_HI}.  This proves \eqref{eq:mple_equation_HIII} and concludes the proof. 
 \end{proof}

\section{Proof of Theorem \ref{thm:beta_H}}
\label{sec:pf_beta_H}

This section focuses on 
the proof of Theorem \ref{thm:beta_H}, which  is presented in three parts: In Section \ref{sec:pf_gradient_H} below we show the concentration of the gradient of the log pseudo-likelihood. The strong concavity of the Hessian of the log pseudo-likelihood is shown in Section \ref{sec:pf_Hessian_H}. These estimates are used to complete the proof of the theorem in Section \ref{sec:pf_H}. Throughout, for notational simplicity, we will just write
  $L_{\vec{\sigma}}$  and $M_{\vec{\sigma}}$ for $L^N_{\vec{\sigma}}$ and $M^N_{\vec{\sigma}}$, respectively.

\subsection{Concentration of the Gradient}
\label{sec:pf_gradient_H}

The first step in the proof of Theorem \ref{thm:beta_H} is to show that the gradient of the log pseudo-likelihood $\bm L_{\vec \sigma}(\bm \beta):=(L^{(1)}_{\vec \sigma}(\bm \beta),  \ldots, L^{(q-1)}_{\vec \sigma}(\bm \beta))'$, as defined in \eqref{eq:mple_equation_HII},   concentrates about zero at the true parameter value $\bm \beta$. 

\begin{lem}\label{lm:2moment_H} Let $\bm L_{\vec \sigma}(\cdot)$ be as in \eqref{eq:mple_equation_HI}. Then under the assumptions of Theorem \ref{thm:beta_H},  
$$\E^N_{\bm \beta}\left[|| \bm L_{\vec \sigma}(\bm \beta)||^2_2\right] \lesssim_{\bm \beta} \frac{1}{N}.$$
\end{lem}

\begin{proof}
  Fix $N \in \N$, recall $L_{\vec \sigma} =  L_{\vec \sigma}^N$  and also denote $\E_{\bm \beta} = E^N_{\bm \beta}$.  
As in the proof of Lemma \ref{lm:2moment}, for $\vec \tau, \vec \tau^* \in \{0, 1\}^N$, define $\newF(\vec \tau, \vec \tau^*) := \frac{1}{N} \sum_{i=1}^N(\tau_i-\tau_i^*)$. Also, choose a coordinate $I$ uniformly at random and replace the $I$-th coordinate of the vector $\vec \sigma \sim \P^N_{\bm \beta}$ by a sample drawn from the conditional distribution of $\sigma_I$ given $(\sigma_v)_{v \ne I}$, and denote the resulting vector by $\sigma'$. For $0 \leq t \leq q-1$, and $J \in \{1, \ldots, N\}$, define 
$$\vec{\sigma}^{(J, t)}  :=  (\sigma_1, \ldots, \sigma_{J-1}, t, \sigma_{J+1}, \ldots, \sigma_N).$$ Now, fix $0 \leq r \leq q-1$. Denoting $\bm \chi_{\vec \sigma, r}=(\chi_{\vec{\sigma}, w}(r))_{1 \leq w \leq N}$,
where recall  $\chi_{\vec{\sigma}, w}(r) = \bm 1\{\sigma_w =r \}$, note that $\newF(\bm \chi_{\vec \sigma, r}, \bm \chi_{\vec \sigma', r})= \chi_{\vec{\sigma}, I}(r) - \chi_{\vec{\sigma}', I}(r)$. Therefore, using \eqref{eq:sigma_r}, \eqref{eq:count_sigma} and \eqref{eq:mple_equation_HI} in the second and third equalities below, we have 
\begin{align}\label{eq:mple_solution_H}
 f(\vec \sigma) & :=\E_{\bm \beta} \left[ \newF(\bm \chi_{\vec \sigma, r}, \bm \chi_{\vec \sigma', r}) | \vec \sigma \right] \nonumber \\  
  & =\frac{1}{N}  \sum_{u=1}^N \left\{ \chi_{\vec \sigma, u}(r)  - \E_{\bm \beta}
  \left[  \chi_{\vec \sigma, u}(r)  |(\sigma_v)_{v \ne u}\right]\right\} \nonumber \\
&=\frac{1}{N}\sum_{u=1}^N \left\{ \bm 1\{\sigma_u =r \} -  \frac{  e^{\beta_r + Q_{G_N, H, \vec \sigma}(u, r)} }{ \sum_{s=0}^{q-1} e^{\beta_s + Q_{G_N, H, \vec \sigma}(u, s)}} \right\} \nonumber \\ 
&= L^{(r)}_{\vec \sigma}(\bm \beta). 
\end{align}
First, multiplying both sides  by $f(\vec \sigma) =L^{(r)}_{\vec \sigma}(\bm \beta)$, 
 next taking expectations, and then using the exchangeability of $(\vec \sigma, \vec \sigma')$
and the antisymmetry of $\newF$, it follows that 
\begin{align}\label{eq:mple_H_I}
   \E_{\bm \beta}\left[L^{(r)}_{\vec \sigma}(\bm \beta)^2\right] =\E_{\bm \beta}\left[f(\vec \sigma)^2\right]  =\frac{\E_{\bm \beta}[(f(\vec \sigma)-f(\vec \sigma')) \newF(\bm \chi_{\vec \sigma, r}, \bm \chi_{\vec \sigma', r})]}{2}.
\end{align}
For $t \in \{0, 1, \ldots, q-1\}$ and $J \in \{1, \ldots, N\}$, 
define
\[ p_{J, t}(\vec \kappa):=\P^N_{\bm \beta}(\sigma'_J= t |\vec \sigma=\vec \kappa, I=J).\]
Then  
\begin{align}\label{eq:mple_H_II}
 \E_{\bm \beta}\left[(f(\vec \sigma)-f(\vec \sigma')) \newF(\bm \chi_{\vec \sigma, r}, \bm \chi_{\vec \sigma', r}) |\vec \sigma\right] &=\frac{1}{N}  \sum_{J=1}^N \sum_{t=0}^{q-1}(f(\vec \sigma)-f(\vec{\sigma}^{(J, t)})) \newF(\bm \chi_{\vec \sigma, r},  \bm \chi_{\vec{\sigma}^{(J, t)}, J}(r)) p_{L, t}(\vec \sigma) \nonumber \\ 
&=\frac{1}{N}  \sum_{J=1}^N  \sum_{t=0}^{q-1} (f(\vec \sigma)-f(\vec{\sigma}^{(J, t)}))  \bar{\chi}_{\vec \sigma, J, t}(r) p_{J, t}(\vec \sigma) ,
\end{align}
where $\bar{\chi}_{\vec \sigma, J, t}(r) := \chi_{\vec \sigma, J}(r) - \chi_{\vec{\sigma}^{(L, t)}, J}(r)  \in \{-1, 1\}$.

Now, recalling from  \eqref{eq:Q_sigma} that $Q_{G_N, H, \vec \sigma}(u, s) \in \{0, -\infty\}$,
we can rewrite \eqref{eq:mple_solution_H} as 
\begin{align*}
f(\vec \sigma) & =L^{(r)}_{\vec \sigma}(\bm \beta) \nonumber \\ 
&= \frac{1}{N} \sum_{u=1}^N \chi_{\vec \sigma, u}(r)  -  \frac{1}{N} \sum_{u=1}^N \left\{ \frac{  e^{\beta_r} \bm 1\{Q_{G_N, H, \vec \sigma}(u, r)=0\} }{  \sum_{s=0}^{q-1} e^{\beta_s} \bm 1\{Q_{G_N, H, \vec \sigma}(u, s)=0\}} \right\}.
\end{align*}
Noting that $\sum_{u=1}^N \chi_{\vec \sigma, u}(r) - \sum_{u=1}^N \chi_{\vec{\sigma}^{(J, t)}, u}(r)=\bar{\chi}_{\vec \sigma, J, t}(r)$, this implies 
\begin{align}\label{eq:mple_H_III}
f(\vec \sigma)-f(\vec \sigma^{(J, t)}) =\frac{1}{N} \left\{\bar{\chi}_{\vec \sigma, J, t}(r) -  (\theta_{\vec \sigma}(r) - \theta_{\vec \sigma^{(J, t)}}(r) ) \right\}, 
\end{align} 
where 
\begin{align}\label{eq:theta_sigma}
\theta_{\vec \sigma}(r) :=\sum_{u=1}^N \left\{ \frac{  e^{\beta_r} \bm 1\{Q_{G_N, H, \vec \sigma}(u, r)=0\} }{  \sum_{s=0}^{q-1} e^{\beta_s} \bm 1\{Q_{G_N, H, \vec \sigma}(u, s)=0\}} \right\}.
\end{align}

Combining \eqref{eq:mple_H_I}, \eqref{eq:mple_H_II}, and \eqref{eq:mple_H_III}, we obtain 
\begin{align}\label{eq:T12_H}
\E_{\bm \beta}[f(\vec \sigma)^2] & = \tfrac{1}{2}(T_1 - T_2) , 
\end{align}
where, for $i = 1, 2,$ $T_i = T_i^N(\vec \sigma)$ is given by 
\begin{align*} 
T_1 & := \E_{\bm \beta}\left[ \frac{1}{N^2}\sum_{J=1}^N  \sum_{t=0}^{q-1} \bar{\chi}_{\vec \sigma, J, t}(r)^2 p_{J, t}(\vec \sigma) \right] \nonumber \\ 
T_2 & :=  \E_{\bm \beta}\left[ \frac{1}{N^2} \sum_{J=1}^N   \sum_{t=0}^{q-1} (\theta_{\vec \sigma}(r) - \theta_{\vec \sigma^{(J, t)}}(r)) \bar{\chi}_{\vec \sigma, J, t}(r) p_{J, t}(\vec \sigma) \right] . 
\end{align*}
Now, using $\bar{\chi}_{\vec \sigma, J, t}(r)^2=1$ and $p_{J, t}(\vec \sigma) \leq 1$, gives $$T_1 \lesssim_{q}   \frac{1}{N}.$$ Moreover, note from  \eqref{eq:Q_sigma} that if $u \notin \cN_{G_N}(J)$, then $\bm 1\{Q_{G_N, H, \vec \sigma}(u, t)=0\} = \bm 1\{Q_{G_N, H, \vec \sigma^{(J, t)}}(u, t)=0\}$, for all $0 \leq t \leq q-1$.  This implies $|\theta_{\vec \sigma}(r) - \theta_{\vec \sigma^{(J, t)}}(r)| \lesssim_{\bm \beta, q}  (d_J(G_N)+1)$,  where recall $d_J(G_N)$ is the degree of the vertex $J$ in  $G_N$. This implies, 
$$T_2  \lesssim_{\bm \beta, q}    \frac{1}{N^2} \sum_{J=1}^N d_J(G_N) \lesssim\frac{1}{N}.$$ 
Therefore, from \eqref{eq:T12_H}, $\E_{\bm \beta}[f(\vec \sigma)^2] \lesssim_{\bm \beta, q}  \frac{1}{N}$, completing the proof of the lemma. 
\end{proof}

\subsection{Strong Concavity of the Hessian}
\label{sec:pf_Hessian_H}

We now show the strong concavity of the Hessian of the log-pseudolikelihood function. For this, recall,
from \eqref{eq:mple_equation_HIII} and \eqref{eq:mple_equation_HI} that
$$ \grad M_{\vec \sigma}(\bm b) = N \bm L_{\vec \sigma}(\bm b):=(L^{(1)}_{\vec \sigma}(\bm b),  \ldots, L^{(q-1)}_{\vec \sigma}(\bm b))',$$
and,  for $1 \leq r \leq q-1$,  
\begin{align}\label{eq:mple_equation_gradient} 
   \frac{1}{N} \frac{\partial}{\partial b_r} M_{\vec \sigma} (\bm b) & = L^{(r)}_{\vec \sigma}(\bm b)  \nonumber \\ 
  & = \frac{\cons_r(\vec \sigma)}{N}  -  \frac{1}{N} \sum_{u=1}^N \left\{ \frac{  e^{b_r} \bm 1\{Q_{G_N, H, \vec \sigma}(u, r)=0\} }{  \sum_{s=0}^{q-1} e^{b_s} \bm 1\{Q_{G_N, H, \vec \sigma}(u, s)=0\}} \right\}, 
\end{align} 
with $a_r$ and $Q_{G_N, H, \vec \sigma}$ defined as in \eqref{eq:count_sigma} and \eqref{eq:Q_sigma}, respectively. Then the scaled Hessian matrix  of $\grad M_{\vec \sigma}(\bm b)$ is
\begin{equation*} 
  \grad \bm L_{\vec \sigma}(\bm b) = \left(\left(\frac{1}{N}\frac{\partial^2}{\partial b_r \partial b_s}M_{\vec \sigma}(\bm b)\right)\right)_{1 \leq r, s \leq q-1}.
\end{equation*}
From \eqref{eq:mple_equation_gradient}, it is easy to check that the $r$-th diagonal element of the scaled Hessian matrix is 
\begin{align}\label{eq:LderivativeI}
\frac{1}{N}  \frac{\partial^2}{\partial^2 b_r}  M_{\vec \sigma}(\bm b) = -  \frac{1}{N} \sum_{u=1}^N   \theta_{\vec \sigma}(u, r) \left( 1-  \theta_{\vec \sigma}(u, r)\right), 
\end{align}
where, for $1 \leq r \ne r' \leq q-1$, 
\begin{equation}
  \label{def:thetaur}
  \theta_{\vec \sigma}(u, r)  :=  \frac{ e^{b_r} \bm 1\{Q_{G_N, H, \vec \sigma}(u, r)=0\} }{  \sum_{s=0}^{q-1} e^{b_s} \bm 1\{Q_{G_N, H, \vec \sigma}(u, s)=0\} }.
  \end{equation}
 Note that $\sum_{u=1}^N \theta_{\vec \sigma}(u, r) = \theta_{\vec \sigma}(r)$, with $\theta_{\vec \sigma}(r)$  defined as in \eqref{eq:theta_sigma}.  Moreover, the $(r, r')$-th element is, 
\begin{align}\label{eq:LderivativeII}
\frac{1}{N}\frac{\partial^2}{\partial b_{r'} \partial b_r}  M_{\vec \sigma}(\bm b)  
&=    \frac{1}{N} \sum_{u=1}^N    \theta_{\vec \sigma}(u, r) \theta_{\vec \sigma}(u, r') ,
\end{align}
for $1 \leq r \ne r' \leq q-1$. To write \eqref{eq:LderivativeI} and \eqref{eq:LderivativeII} more compactly, define the $(q-1)$-dimensional vector $\cO_{\vec \sigma}(u) := (\theta_{\vec \sigma}(u, 1), \ldots, \theta_{\vec \sigma}(u, q - 1))^{\top}$. Then, the identities in \eqref{eq:LderivativeI} and \eqref{eq:LderivativeII} can be written more succinctly as  
\begin{align}\label{eq:L_beta}
-\grad \bm L_{\vec \sigma}(\bm b) = \frac{1}{N} \sum_{u=1}^N  \Big\{   \mathrm{diag}(\cO_{\vec \sigma}(u) ) - \cO_{\vec \sigma}(u) \cO_{\vec \sigma}^{\top} (u)  \Big\}, 
\end{align}
where for any matrix $W$, $W^{\top}$ denotes its transpose. Therefore, to show that $M_{\vec \sigma}(\bm b)$ is strongly concave at the true value $\bm b = \bm \beta$ with high probability, it suffices to prove that the minimum eigenvalue $\lambda_{\mathrm{min}}(-\grad \bm L_{\vec \sigma}(\bm \beta) ) $ is strictly positive  with probability going to 1. This is proved in the following lemma:

\begin{lem}\label{lm:condition_H}
Under the assumptions of Theorem \ref{thm:beta_H}, 
$$\lim_{\varepsilon \rightarrow 0}  \lim_{N \rightarrow \infty} \P^N_{\bm \beta}(\lambda_{\mathrm{min}}(-\grad \bm L_{\vec \sigma}(\bm \beta) ) > \varepsilon ) = 1.$$
\end{lem}  
\begin{proof}  Fix $\varepsilon > 0$, $1 \leq r \leq q-1$, and define  
$$\cA_{\varepsilon, r} := \left\{ \vec \sigma \in \{0, 1, \ldots, q-1\}^N: \ucons_r(\vec \sigma) > \varepsilon N  \right \},$$
and 
\begin{align}\label{eq:ucons_epsilon}
\cA_{\varepsilon} := \bigcap_{r=1}^{q-1} \cA_{\varepsilon, r}.
\end{align} 
Now, recalling  \eqref{def:thetaur} and  \eqref{eq:Q_sigma}, and using 
$$\sum_{s=0}^{q-1} e^{\beta_s} \bm 1\{Q_{G_N, H, \vec \sigma}(u, s)=0\}   \leq \sum_{s=0}^{q-1} e^{\beta_s},$$ for $1 \leq u \leq N$, gives  
\begin{align}\label{eq:bd_I}
  \sum_{u=1}^N  \theta_{\vec \sigma}(u, r) 
  & =    \sum_{u=1}^N  \frac{ e^{\beta_r} \bm 1\{Q_{G_N, H, \vec \sigma}(u, r)=0\} }{  \sum_{s=0}^{q-1} e^{\beta_s} \bm 1\{Q_{G_N, H, \vec \sigma}(u, s)=0\} } \nonumber \\  
& \geq  \calJ(\bm \beta)  \sum_{u=1}^{N} \bm 1\{  \vec \sigma_{\cN_{G_N}(u)} \subseteq \cN_{H}(r) \} \nonumber \\  
& =  \calJ(\bm \beta)  \ucons_r(\vec \sigma) , 
\end{align} 
where $$\calJ(\bm \beta) :=  \min_{1 \leq r \leq q-1} \left\{ \frac{e^{\beta_r}}{ \sum_{s=0}^{q-1} e^{\beta_s}} \right \} > 0 .$$ Denote $\varepsilon_0 := \varepsilon \calJ(\bm \beta) $. Also, given $\vec\sigma  \in \{0, 1, \ldots, q-1\}^N$, define 
\begin{align}\label{eq:cepsilon}
\cT_{\varepsilon_0}(\vec \sigma) & :=\Big\{u \in V(G_N): \theta_{\vec \sigma}(u, r) > \tfrac{\varepsilon_0}{2},~\forall~ 0 \leq r \leq q-1 \Big\}, 
\end{align}
with $\theta_{\vec \sigma}$ as in  \eqref{def:thetaur}. 
Then, for $\vec \sigma \in \cA_\varepsilon$, using \eqref{eq:bd_I} and the bound $\theta_{\vec \sigma}(u, r)  \leq 1$ on $\cT_{\varepsilon_0}(\vec \sigma)$, we conclude that 
$$\varepsilon_0 N < \sum_{u=1}^N \theta_{\vec \sigma}(u, r) \leq  \tfrac{1}{2} \varepsilon_0 N  + |\cT_{\varepsilon_0} (\vec \sigma) |,$$
which implies $|\cT_{\varepsilon_0} (\vec \sigma) | \geq \frac{\varepsilon_0 N}{2}$. 

Now, recalling  that each of the matrices in the representation of $\grad \bm L_{\vec \sigma}(\bm \beta)$ in  \eqref{eq:L_beta} is symmetric, an application of 
Weyl's inequality yields  
\begin{align} 
  \lambda_{\mathrm{min}}(-\grad \bm L_{\vec \sigma}(\bm \beta) )   & \geq \frac{1}{N} \sum_{u=1}^N  \lambda_{\mathrm{min}} \Big\{   \mathrm{diag}(\cO_{\vec \sigma}(u) ) - \cO_{\vec \sigma}(u) \cO^{\top}_{\vec \sigma}(u)  \Big\}. \label{eq:bd_II_eigenvalue}
\end{align} 
We now invoke the following result: 
\begin{lem}\label{lm:Lepsilon} Let $\cT_{\varepsilon_0}(\vec \sigma)$ be as defined in  \eqref{eq:cepsilon}.  If $u \in \cT_{\varepsilon_0}(\vec \sigma)$, then 
$$\lambda_{\mathrm{min}}\Big\{ \mathrm{diag}(\cO_{\vec \sigma}(u) ) -  \cO_{\vec \sigma}(u) \cO^{\top}_{\vec \sigma}(u)  \Big\} \geq \tfrac{\varepsilon_0}{2}.$$
\end{lem}
Deferring the proof of this lemma to later, we first use  it to  complete the proof of Lemma \ref{lm:condition_H}. 
Note that for $\vec \sigma \in \cA_\varepsilon$, Lemma \ref{lm:Lepsilon}, 
when combined with the fact that the matrix $\mathrm{diag}(\cO_{\vec \sigma}(u) ) - \cO_{\vec \sigma}(u) \cO^{\top}_{\vec \sigma}(u)$
is non-negative definite, for $1 \leq u \leq N$, and the bound $|\cT_{\varepsilon_0} (\vec \sigma) | \geq \frac{\varepsilon_0 N}{2}$,  implies
that 
\begin{align} \label{eq:bd_II}  
  \lambda_{\mathrm{min}}(-\grad \bm L_{\vec \sigma}(\bm \beta) ) 
  & \geq  \frac{1}{N} \sum_{u \in \cT_{\varepsilon_0}(\vec \sigma)}  \lambda_{\mathrm{min}} \Big\{   \mathrm{diag}(\cO_{\vec \sigma}(u) ) - \cO_{\vec \sigma}(u) \cO^{\top}_{\vec \sigma}(u)  \Big\}  
   \nonumber \\ 
  & \geq \frac{\varepsilon_0  | \cT_{\varepsilon_0}(\vec \sigma)| }{2 N }  \nonumber \\ 
  & \geq \frac{\varepsilon_0^2 }{4}. 
  \end{align}  
Moreover,  note that the rainbow condition \eqref{eq:sigma_neighborhood}  
implies that 
$$\lim_{\varepsilon \rightarrow 0} \liminf_{N \rightarrow \infty} \P^N_{\bm \beta}(\cA_{\varepsilon,r}) = 1$$
for all $1 \leq r \leq q-1$, 
and hence, the relation  \eqref{eq:ucons_epsilon} implies
 $$\lim_{\varepsilon \rightarrow 0} \liminf_{N \rightarrow \infty} \P^N_{\bm \beta}(\cA_\varepsilon) = 1.$$
Then, for every $\delta > 0$ we can choose $\varepsilon = \varepsilon(\bm \beta, \delta) >0 $ large enough such that $\P^N_{\bm \beta}(\cA_\varepsilon) > 1-\delta$. Hence,  it follows from \eqref{eq:bd_II} that $$\P^N_{\bm \beta}\left( \lambda_{\mathrm{min}}(-\grad \bm L_{\vec \sigma}(\bm \beta) )  \geq \frac{\varepsilon_0^2}{4} \right) > 1-\delta.$$ This completes the proof of the lemma. 
\end{proof}

To complete the justification of the last proof, we now prove  Lemma \ref{lm:Lepsilon}. 

\begin{proof}[Proof of Lemma \ref{lm:Lepsilon}]  It is immediate from the definition in \eqref{def:thetaur} that 
  $\sum_{r=0}^{q-1} \theta_{\vec \sigma}(u, r)=1$, for $1 \leq u \leq N$. This implies that  for $1 \leq u \leq N$, the matrix $ \mathrm{diag}(\cO_{\vec \sigma}(u) ) -  \cO_{\vec \sigma}(u) \cO^{\top}_{\vec \sigma}(u)$ is the covariance matrix of a multinomial distribution with probability vector $\cO_{\vec \sigma}(u)$. It is well-known \cite[Theorem 2.2]{multinomial_wn} that any  eigenvalue $\lambda$ of the covariance matrix of a multinomial distribution satisfies the equation,  
\begin{align}\label{eq:lambda_u}
\sum_{r=1}^{q-1}\frac{\theta_{\vec \sigma}(u, r)^2}{\theta_{\vec \sigma}(u, r) - \lambda} = 1. 
\end{align} 
This implies, using the bounds $\varepsilon_0/2 \leq \theta_{\vec \sigma}(u, r) \leq 1$, for $u \in \cT_{\varepsilon_0}(\vec \sigma)$, 
$$1= \sum_{r=1}^{q-1}\frac{\theta_{\vec \sigma}(u, r)^2}{\theta_{\vec \sigma}(u, r) - \lambda} \leq  \frac{1}{\frac{\varepsilon_0}{2} - \lambda} \sum_{r=1}^{q-1} \theta_{\vec \sigma}(u, r) <  \frac{1}{\frac{\varepsilon_0}{2} - \lambda},$$
that is, $\lambda > \varepsilon_0/2$. This completes the proof. 
\end{proof}

\subsection{Completing the Proof of Theorem \ref{thm:beta_H}}
\label{sec:pf_H}

Equipped with Lemmas \ref{lm:2moment_H} and \ref{lm:condition_H}, the proof of $\sqrt N$-consistency can be completed as follows: For $t \in (0, 1)$, define 
$$ \bm \beta^N_t := t \bm {\hat \beta}^N + (1- t) \bm \beta \quad \mbox{ and } \quad h_N(t) := (\hat{\bm \beta}^N-\bm \beta)^\top  \bm L_{\vec \sigma}(\bm \beta_t), 
$$ 
where $L_{\vec \sigma} = L_{\vec \sigma}$ is as defined in \eqref{eq:mple_equation_HII}-\eqref{eq:mple_equation_HI}. Then note that for $t \in (0, 1)$,
\begin{align*}
-h'_N(t) & = -(\hat{\bm \beta}^N-\bm \beta)^\top \grad L_{\vec \sigma}(\bm \beta^N_t) (\hat{\bm \beta}^N-\bm \beta) \nonumber \\ 
& \geq \lambda_\mathrm{min}(- \grad \bm L_{\vec \sigma}(\bm \beta^N_t)) ||\hat{\bm \beta}^N-\bm \beta||_2^2 \nonumber \\ 
& \geq 0, 
\end{align*} 
where the last inequality holds because by \eqref{eq:bd_II_eigenvalue}, 
$\lambda_\mathrm{min}(- \grad \bm L_{\vec \sigma}(\tilde{\bm \beta})) \geq 0$, for all $\tilde{\bm \beta} \in \R^{q-1}$, since the matrix $\mathrm{diag}(\cO_{\vec \sigma}(u) ) - \cO_{\vec \sigma}(u) \cO^{\top}_{\vec \sigma}(u)$ is non-negative definite, for $1 \leq u \leq N$. 
Now, define $Z_N := ||\hat{\bm \beta}^N-\bm \beta||_2 $.  Then, for $K \in (0,\infty)$ fixed, by the Cauchy-Schwarz inequality, and the
fact that $\bm L_{\vec \sigma}(\hat{\bm \beta}^N) = 0$, 
\begin{align} 
Z_N \cdot || \bm L_{\vec \sigma}(\bm \beta)||_2   = ||\hat{\bm \beta}^N-\bm \beta||_2 \cdot || \bm L_{\vec \sigma}(\bm \beta)||_2 & \geq |(\hat{\bm \beta}^N-\bm \beta)^\top \bm L_{\vec \sigma}(\bm \beta)| \nonumber \\
&=|h_N(1)-h_N(0)|\\ 
&=\left|-\int_0^1 h_N'(t) \mathrm dt \right|  \nonumber \\ 
& \geq \left|-\int_{0}^{\min\{1, \frac{K}{Z_N} \} } h_N'(t) \mathrm dt \right|  \nonumber \\ 
\label{eq:beta_Lsigma} & \geq  Z_N^2 \int_{0}^{\min\{1, \frac{K}{Z_N} \} }   \lambda_\mathrm{min}(- \grad \bm L_{\vec \sigma}(\bm \beta_t)) \mathrm dt . 
\end{align} 

Next, for constants $\varepsilon > 0$ and $C < \infty$, define, 
\begin{align*} 
T_{N, C, \varepsilon} & = \Bigg \{\vec \sigma: || \bm L_{\vec \sigma}(\bm \beta)||_2 \leq \frac{C}{\sqrt N}  \text{ and }  \min_{\bm b: |||\bm b-\bm \beta||_2 \leq K}\lambda_\mathrm{min}(- \grad \bm L_{\vec \sigma}(\bm b)) \geq \varepsilon  \Bigg \}.
\end{align*} 
Then by Lemma \ref{lm:2moment_H} and Lemma \ref{lm:condition_H}, for $0 < \delta < 1$ fixed, there exist 
$\varepsilon := \varepsilon(\delta, \bm \beta) > 0$ and  $C := C(\delta, \bm \beta) <  \infty$ such that $\P^N_{\bm \beta }(T_{N, C, \varepsilon}) \geq 1-\delta$.  Moreover,  since $||\bm \beta^N_t - \bm \beta ||_2 = t ||\hat{\bm \beta}^N - \bm \beta ||_2 \leq K$, for $\vec \sigma \in T_{N,C,\varepsilon}$, 
$$\int_{0}^{\min\{1, \frac{K}{Z_N} \} }   \lambda_\mathrm{min}(- \grad \bm L_{\vec \sigma}(\bm \beta^N_t)) \mathrm dt \geq \varepsilon \min\left\{1, \frac{K}{Z_N} \right\}.$$ 
When combined with \eqref{eq:beta_Lsigma}, and the definition of $T_{N,C,\varepsilon}$, this implies 
\begin{align}\label{eq:estimate_difference}
\min\left\{Z_N, K \right\} = Z_N  \min\left\{1, \frac{K}{Z_N} \right\} \leq \frac{C}{\varepsilon \sqrt N}. 
\end{align} 
Since $K$ is fixed, this shows $Z_N$ converges in probability to zero. Therefore, for $N$ large enough with probability $1-\delta$, we have $\min\left\{Z_N, K \right\}= Z_N$. Hence, \eqref{eq:estimate_difference} shows that $Z_N=  ||\hat{\bm \beta}-\bm \beta||_2  = O_{\delta, \bm \beta}(1/\sqrt N)$ with probability at least $1- 2 \delta$. This completes the proof of Theorem \ref{thm:beta_H}.  \hfill $\Box$ \\

\begin{remark}\label{remark:estimate}  {\em (Computation of the MPL estimate) The MPL estimate of $\bm \beta$ in an $H$-coloring model can be computed using a simple gradient descent algorithm. To this end,  recall that the goal is to minimize the function $-M_{\vec \sigma}(\bm b)$, as defined in \eqref{eq:mple_function},  over $\bm b = (b_1, b_2, \ldots, b_{q-1})' \in \R^{q-1}$. Hence, starting with some $\bm \beta^{(1)} = (b_1^{(1)}, b_2^{(1)}, \ldots, b_{q-1}^{(1)})'$, define, for $t \geq 1$, the $(t+1)$-th step of the gradient descent algorithm with step size $\gamma$ as, 
\begin{align}\label{eq:beta_gd}
\bm \beta^{(t+1)} := \bm \beta^{(t)} + \gamma  \grad \left(\tfrac{1}{N}M_{\vec \sigma}(\bm b) \right) \Big|_{\bm b =  \bm \beta^{(t)}}. 
\end{align} 
To show the convergence of this algorithm it suffices to prove that the function $-\frac{1}{N}\grad M_{\vec \sigma}(\bm b) = -  \bm L_{\vec \sigma}(\bm b)$ is Lipschitz. For this, note from \eqref{eq:L_beta} and  Weyl's inequality, that with $\lambda_{\mathrm{max}}$ denoting  the largest eigenvalue,  
\begin{align*}
\lambda_{\mathrm{max}}(-\grad \bm L_{\vec \sigma}(\bm b) )  \leq  \frac{1}{N} \sum_{u=1}^N  \lambda_{\mathrm{max}} \Big\{   \mathrm{diag}(\cO_{\vec \sigma}(u) ) - \cO_{\vec \sigma}(u) \cO_{\vec \sigma}(u)'  \Big\}. 
\end{align*} 
Now, recalling $\cO_{\vec \sigma}(u) = (\theta_{\vec \sigma}(u, 1), \ldots, \theta_{\vec \sigma}(u, q-1))^{\top}$, the definition in \eqref{def:thetaur}, and using the bound $\theta_{\vec \sigma}(u, r) \leq 1$ in \eqref{eq:lambda_u} gives,
\begin{align*} 
 \lambda_{\mathrm{max}} \Big\{   \mathrm{diag}(\cO_{\vec \sigma}(u) ) - \cO_{\vec \sigma}(u) \cO^{\top}_{\vec \sigma}(u)  \Big\} & \leq 1- \sum_{r=1}^{q-1} \theta_{\vec \sigma}(u, r)^2 \nonumber \\  
 & < 1.
\end{align*} 
Denote by $\grad^2 M _{\vec \sigma}(\bm b)$ the Hessian matrix of $M _{\vec \sigma}(\bm b)$. Then, the last two displays together imply $\lambda_{\mathrm{max}}(-  \frac{1}{N}\grad^2 M _{\vec \sigma}(\bm b) ) = \lambda_{\mathrm{max}}(-  \grad \bm L_{\vec \sigma}(\bm b) ) \leq 1$, for all $\bm b \in \R^{q-1}$, which shows that the function $-\grad M_{\vec \sigma}(\bm b)$ is 1-Lipschitz. 

Moreover, by Lemma \ref{lm:condition_H}, for every $0 < \delta < 1$ there exists an $\varepsilon:=\varepsilon (\delta, \bm \beta) > 0$ such that the function $- \frac{1}{N} M_{\vec \sigma}(\bm b)$ is $\varepsilon$-strongly convex with probability at least $1-\delta/2$. Therefore, using the gradient descent algorithm \eqref{eq:beta_gd} with step size $\gamma=1$ gives, by \cite[Theorem 3.10]{optimization_notes},
$$ ||\bm \beta^{(t+1)} - \hat{\bm \beta}^N ||_2    \leq  e^{-\varepsilon t} ||\bm \beta^{(1)} - \hat{\bm \beta}^N ||_2 ,$$ 
probability at least $1-\delta$. Hence, after $t \geq \frac{\log (N ||\bm \beta^{(1)} - \hat{\bm \beta}^N ||_2)}{\varepsilon}$ steps, 
$$||\bm \beta^{(t+1)} - \hat{\bm \beta}^N ||_2 \leq  \frac{1}{\sqrt N},$$ with probability at least $1-\delta$. 
Hence, by Theorem \ref{thm:beta_H}, after $t=O_{\bm \beta, \bm \beta^{(1)}}(\log N)$ steps of the gradient descent algorithm, $ ||\bm \beta^{(t)} -  \bm \beta ||_2 = O(1/\sqrt N)$ with probability at least $1-\delta$. } 
\end{remark}

\section{Proof of Corollary \ref{cor:unconstrained_H}} 
\label{sec:pfcorollary}

In this section we prove Corollary \ref{cor:unconstrained_H}. We begin in
Section \ref{sec:pf_expectation} with a useful sufficient condition for \eqref{eq:sigma_neighborhood}, which is expressed in terms of the expected number of vertices in $G_N$ with color $r$ (recall Remark \ref{remark1}). We then use this result to prove Corollary \ref{cor:unconstrained_H} in Section \ref{sec:pfunconstrained_H}.

\subsection{A Useful Sufficient Condition}
\label{sec:pf_expectation}

For $0 \leq r \leq q-1$, recall that $\cons_r(\vec \sigma) = \sum_{u=1}^N \bm 1\{\sigma_u = r\}$ counts the number of vertices with color $r$. The following proposition shows that the rainbow condition \eqref{eq:sigma_neighborhood} holds, whenever $\cons_r(\vec \sigma)$,  is $\Theta(N)$ with high probability in a neighborhood of $\bm \beta$. To this end, for $\delta > 0$, denote by $\bm \beta_{\delta, r} := (\beta_1, \ldots, \beta_{r-1}, \beta_r-\delta,  \beta_{r+1}, \ldots, \beta_{q-1})$.

\begin{ppn}\label{ppn:expectation} Fix $q \geq 2$, $\bm \beta = (\beta_1, \beta_2, \ldots, \beta_{q-1})' \in \R^{q-1}$, and a connected constraint graph $H$. Given a  graph $G_N$ on $N$ vertices, for $N \in \mathbb{N}$,
  such that the sequence $\{G_N\}_{N \in \N}$  has  uniformly
  bounded maximum degree, and $\vec \sigma \sim \P^N_{\bm \beta}$ from the
  $H$-coloring model \eqref{eq:model_H}, assume that,  for some $1 \leq r \leq q-1$, 
  \begin{align}
    \label{suff:cond}
    \lim_{\delta \rightarrow 0} \liminf_{N \rightarrow \infty} \frac{1}{N} \E^N_{\bm \beta_{\delta, r}}[\cons_r(\vec \sigma)] > 0. 
  \end{align}
  Then  
\begin{align}\label{eq:prob_ar}
\lim_{\varepsilon \rightarrow 0} \limsup_{N\rightarrow\infty}\P^N_{\bm \beta}(\cons_r(\vec \sigma) \leq \varepsilon N)=0. 
\end{align}
This implies \eqref{eq:sigma_neighborhood} holds whenever \eqref{suff:cond} does for all $1 \leq r \leq q-1$. 
\end{ppn}

\begin{proof} Recall the definition of the log-partition function for the $H$-coloring model from \eqref{eq:model_H_II}. Then note that 
\begin{align*} 
F_{G_N}^{(r)} (\bm \beta) & : = \frac{\partial}{\partial \beta_r} F_{G_N}(\bm \beta) 
=\E^N_{\bm \beta}[\cons_r(\vec \sigma)]. 
\end{align*} 
Thus, for any $\delta, \varepsilon >0$, using Jensen's inequality, 
\begin{align}\label{eq:csigma_probability}
\P^N_{\bm \beta}(\cons_r(\vec \sigma)<\varepsilon N)  =\P^N_{\beta}(e^{- \delta \cons_r(\vec \sigma)}>e^{- \delta \varepsilon N}) 
& \le e^{ \delta \varepsilon N} \E^N_{\bm \beta} [e^{-\delta \cons_r(\vec \sigma)}]  \nonumber \\ 
& \le e^{ \delta \varepsilon N+F_{G_N}(\bm \beta-\delta \bm e_r)-F_{G_N}(\bm \beta)},
\end{align} 
where $\bm e_r$ is the $r$-th basis vector in $\R^{q-1}$.  Note that the function $t \mapsto F_{G_N}^{(r)} (\beta_1, \ldots, \beta_{r-1}, t,  \beta_{r+1}, \ldots, \beta_{q-1})$ is monotone, since 
\begin{align*} 
 \frac{\partial }{\partial t} F_{G_N}^{(r)} (\beta_1, \ldots, \beta_{r-1}, t,  \beta_{r+1}, \ldots, \beta_{q-1})  & = \Var_{(\beta_1, \ldots, \beta_{r-1}, t, \beta_{r+1}, \ldots, \beta_{q-1})}^N[\cons_r(\vec \sigma)] \nonumber \\ 
& \geq 0.
\end{align*} 
This implies, taking logarithms on both sides in \eqref{eq:csigma_probability}, 
\begin{align*}
\log \P^N_{\bm \beta}(\cons_r(\vec \sigma)<\varepsilon N)  & \leq \delta \varepsilon N+F_{G_N}(\bm \beta-\delta \bm e_r)-F_{G_N}(\bm \beta) \\ 
& =  \delta \varepsilon N-\int_{\beta_r-\delta}^{\beta_r}F_{G_N}^{(r)} (\beta_1, \ldots, \beta_{r-1}, t,  \beta_{r+1}, \ldots, \beta_{q-1})\mathrm dt \nonumber \\ 
& \le  \delta \varepsilon N - \delta F_{G_N}^{(r)} (\beta_1, \ldots, \beta_{r-1}, \beta_r-\delta,  \beta_{r+1}, \ldots, \beta_{q-1}), 
\end{align*} 
where the last step uses the monotonicity of the function $t \mapsto F_{G_N}^{(r)} (\beta_1, \ldots, \beta_{r-1}, t,  \beta_{r+1}, \ldots, \beta_{q-1})$. Dividing both sides by $N$ and taking first the limit superior  as $N\rightarrow \infty$ and then the limit
  as $\varepsilon \rightarrow 0$, 
\begin{align*}
 \lim_{\varepsilon \rightarrow 0}\limsup_{N\rightarrow\infty}  \frac{1}{N}\log \P^N_{\bm \beta}(\cons_r(\vec \sigma)<\varepsilon N) 
& \le   - \delta \liminf_{N\rightarrow\infty}\frac{1}{N}{F_{G_N}^{(r)} (\beta_1, \ldots, \beta_{r-1}, \beta_r-\delta,  \beta_{r+1}, \ldots, \beta_{q-1})} \nonumber \\ 
& = - \delta \liminf_{N\rightarrow\infty}\frac{1}{N} \E^N_{\bm \beta_{\delta, r}} [\cons_r(\vec \sigma)], 
\end{align*}  which is strictly negative due to  \eqref{suff:cond}. This is easily seen to imply \eqref{eq:prob_ar}.

Finally, note that for every feasible configuration, that is, for every $\vec \sigma$ such that $\P^N_{\bm \beta}(\vec \sigma) > 0$,
    $\ucons_r(\vec \sigma) \geq \cons_r(\vec \sigma)$, for all
$1 \leq r \leq q-1$.   
       Therefore, for every $1 \leq r \leq q-1$, 
\begin{align*}
\P^N_{\bm \beta}\left( \ucons_r(\vec \sigma) > \varepsilon N \right) & \geq \P^N_{\bm \beta}\left( \cons_r(\vec \sigma)  > \varepsilon N  \right).  
\end{align*} 
Taking first the limit inferior as $N \rightarrow \infty$, and then the limit as $\varepsilon \rightarrow 0$ and using
  \eqref{eq:prob_ar}, it follows that the quantity on the last display converges to $1$, thus establishing 
  \eqref{eq:sigma_neighborhood}. 
\end{proof}

\subsection{Proof of Corollary \ref{cor:unconstrained_H}}
\label{sec:pfunconstrained_H}

The proof of this corollary entails verifying condition \eqref{suff:cond} in Proposition \ref{ppn:expectation}. The argument towards this proceeds along similar lines as Lemma \ref{lm:lp}. Hereafter, suppose that the vertex $s \in \{0, 1, \ldots, q-1\}$ in the graph $H$ is unconstrained.  For $1 \leq r \leq q-1$ (note that $r$ can be equal to $s$) and $u \in V(G_N)$, we partition the collection $\Omega_{G_N}(H)$ of all admissible $H$-colorings of $G_N$ as follows:  Let $\Omega_{G_N, u, r}^+(H)$ be the collection of colorings of $G_N$ with $\sigma_u=r$, and $\Omega_{G_N, u, r}^-(H)= \Omega_{G_N}(H) \backslash \Omega_{G_N, u, r}^+(H)$ the collection of colorings of $G_N$ with $\sigma_u \ne r$. Now, define the map $g: \Omega_{G_N, u, r}^-(H) \rightarrow \Omega_{G_N, u, r}^+(H)$, which takes $\vec \sigma \in \Omega_{G_N, u, r}^-(H)$ to $\vec \sigma' \in \Omega_{G_N, u, r}^+(H)$ which is defined as 
$$\sigma'_v=
\left\{
\begin{array}{ccc} 
r  &  \text{ for } v = u,   \\ 
s &  \text{ for } v \in \cN_{G_N}(u),   \\
\sigma_v  &  \text{otherwise.}   \\
\end{array}
\right. $$
(Observe that this is a valid coloring because there is no constraint on coloring a vertex $s$.) Note that given $\vec \sigma' \in \Omega_{G_N, u, r}^+(H)$, $$|g^{-1}(\sigma')|:=|\{\sigma \in  \Omega_{G_N, u, r}^-(H) : g(\vec \sigma)=\vec\sigma'\}| \leq q^{d_u(G_N) + 1}.$$ Moreover, if $\sigma \in g^{-1}(\vec \sigma')$, then for $\bm b \in \R^{q-1}$, 
 $$e^{ \sum_{t=0}^{q-1} b_t \cons_t(\vec\sigma) } \leq e^{|| \bm b ||_1 {(d_u(G_N) + 1)}}  e^{ \sum_{t=0}^{q-1} b_t \cons_s(\vec\sigma') },$$ because $|\cons_t(\vec \sigma') - \cons_t(\vec \sigma)|  \leq    d_u(G_N)+1$, for all $0 \leq t \leq q-1$. Therefore,  it follows that 
\begin{align}\label{eq:Z_sigma}
\sum_{\vec \sigma \in  \Omega_{G_N, u, r}^-(H) } e^{ \sum_{t=0}^{q-1} b_t \cons_t(\vec\sigma) }  & \leq \sum_{\vec \sigma' \in \Omega_{G_N, u, r}^+(H)} \sum_{\vec \sigma \in g^{-1}(\vec \sigma')} e^{ \sum_{t=0}^{q-1} b_t \cons_t(\vec\sigma) } \nonumber \\ 
& \leq  e^{|| \bm b ||_1 {(d_u(G_N) + 1)}}  q^{d_u(G_N) + 1} \sum_{\vec \sigma' \in \Omega_{G_N, u, r}^+(H)} e^{ \sum_{t=0}^{q-1} b_t \cons_t(\vec\sigma') }.  
\end{align}
Hence, 
\begin{align*}
 \E_{\bm b} [\cons_r(\vec \sigma)]  = \sum_{u=1}^N \P^N_{\bm b}(\sigma_u = r ) & =  \sum_{u=1}^N \frac{\sum_{\vec \sigma \in \Omega_{G_N, u, r}^+(H)}  e^{ \sum_{t=0}^{q-1} b_t \cons_t(\vec\sigma) }  }{\sum_{\vec \sigma \in \Omega_{G_N}(H)} e^{ \sum_{t=0}^{q-1} b_t \cons_t(\vec\sigma) } }  \nonumber \\ 
& \geq \sum_{u=1}^N \frac{1}{1+ e^{|| \bm b ||_1 {(d_u(G_N) + 1)}}  q^{d_u(G_N) + 1} } \tag*{(using \eqref{eq:Z_sigma})}\nonumber \\ 
& \geq  \frac{1}{1+ e^{|| \bm b ||_1 (M+1)}  q^{M + 1} } \frac{|\{u: d_u(G_N) \leq M \}|}{N}, 
\end{align*} 
where the last step holds for any $M >0$. Then, by Lemma \ref{lm:M}, choosing $M=M_0$ such that $\liminf_{N \rightarrow \infty} \frac{|\{u: d_u(G_N) \leq M_0\}|}{N} > 0$, it follows that 
\begin{align*}
 \liminf_{N \rightarrow \infty} \frac{1}{N} \E_{\bm b} [\cons_r(\vec \sigma)]  & \geq  \frac{1}{1+ e^{|| \bm b ||_1 (M_0+1)}  q^{M_0 + 1} } \liminf_{N \rightarrow \infty} \frac{|\{u: d_u(G_N) \leq M_0 \}|}{N} > 0,
\end{align*} 
for all $\bm b \in \R^{q-1}$ and all $1 \leq r \leq q-1$.  Then, by Proposition \ref{ppn:expectation} and Theorem \ref{thm:beta_H}, the result in Corollary \ref{cor:unconstrained_H} follows.  \qed

\section{Proof of Corollary \ref{cor:coloring}}
\label{sec:pfcorollary_II}
 
We say a set $A \subset V(G_N)$ is {\it neighborhood disjoint}, if $\cN_{G_N}(v) \bigcap \cN_{G_N}(w) = \emptyset$, for all $v, w \in A$ with $v \ne w$. In other words, a set $A$ is  neighborhood disjoint if the neighborhoods of any two vertices in $A$ are disjoint.  We begin by showing that if a graph has bounded average 2-neighborhood, it has a $\Theta(N)$ size subset of vertices that  is neighborhood disjoint.

\begin{lem}\label{lm:2neighborhood} Let $G_N$ be a graph with $V(G_N)=\{1, 2, \ldots, N\}$ and $W  \subset V(G_N)$, with $|W| \geq \frac{N}{2}$. Then there exists $A \subseteq W$  such that $A$ is neighborhood disjoint in $G_N$ and 
  \begin{align}
    \label{ineq-A}
    |A|  \geq \frac{|W|}{1+\frac{1}{|W|} \sum_{v=1}^{|W|} d_{2,  v}(G_N) } 
    & \geq \frac{N}{4(1+\frac{1}{ N} \sum_{v=1}^{N} d_{2,  v}(G_N))}.
    \end{align}
\end{lem} 

\begin{proof} Let $\pi =(\pi_1, \ldots, \pi_N)$ be a uniformly random permutation of the vertices of $G_N$.  Suppose $1 \leq i_1< i_2 < \ldots i_{|W|} \leq N$ is such that $W=\{\pi(i_1), \pi(i_2), \ldots, \pi(i_{|W|})\}$. We will build a (random) set $\X$ iteratively as follows: Start with the empty set $\X_0: = \emptyset$. In the $w$-th step, for $1 \leq w \leq |W|$, construct the set $\X_w$ by including the vertex $\pi_{i_w}$ into the set $\X_{w-1}$ if none of the vertices in $G_N$ that are at distance one or two from $\pi_{i_w}$ appear before it in $\pi$. Otherwise, let $\X_w= \X_{w-1}$ and we move to step $w+1$. In the end, define $\X:= \X_{|W|}$. Clearly, the set $\X$ constructed in this way is  neighborhood disjoint in $G_N$. Moreover, a vertex $w \in W$ appears in the set $\X$ if and only if it appears in the permutation $\pi$ before all the $d_{2,  w}(G_N)$ vertices that are at distance 1 or 2 from $w$. Note that this happens with probability $1/(1 + d_{2,  w}(G_N))$. Hence, denoting by $d(u, v)$ the shortest path distance between the vertices $u$ and $v$ in $G_N$ and using the fact that  the arithmetic mean is at least the harmonic mean, we see that   
\begin{align*}
  \E \left[|\X|\right]    =
  \sum_{w=1}^{|W|}  \bm 1\left\{ \pi_{i_w} < \min \{\pi_s: 1 \leq d(\pi_{i_w},\pi_{s}) \leq 2 \} \right\} 
  & =
  \sum_{w=1}^{|W|} \frac{1}{1+d_{2,w}(G_N)} \nonumber \\ 
 & \geq \frac{|W|}{1+\frac{1}{|W|} \sum_{v=1}^{|W|} d_{2,  v}(G_N) }.
\end{align*} 
 This implies that there must exist a  neighborhood disjoint set $A \subset W$ for which the
 first inequality in \eqref{ineq-A} holds.  On the other hand, the second inequality in \eqref{ineq-A} is a consequence
 of the fact that  $|W| \geq N/2$ and $|W| + \sum_{v=1}^{|W|} d_{2,  v}(G_N)
 \leq N + \sum_{v=1}^{N} d_{2,  v}(G_N)$, which completes the proof. 
\end{proof}

We now proceed to prove Corollary \ref{cor:coloring}. 

\begin{proof}[Proof of Corollary \ref{cor:coloring}.] 
    Suppose $\{G_N\}_{N \in  \N}$ is a sequence of graphs that satisfies \eqref{deg2-cond},  that is, $\sup_{N \in  \N} \frac{1}{N} \sum_{v=1}^N d_{2,  v}(G_N) < \infty$. For the $q$-coloring model,  fix $r \in \{0, 1, \ldots, q-1\}$, 
  and note that since 
  $H=K_q$ is the complete graph on $q$ vertices, it follows that $\cN_H(r)=\{0, 1, \ldots, q-1\} \backslash\{r\}$.
    Hence, recalling the definition of the set $\cU^N_r(\vec \sigma)$ from \eqref{eq:unconstrained_r}, we see that 
$$\cU^N_r(\vec \sigma):= \{ u \in \{1, \ldots, N\}: r \notin \vec \sigma_{\cN_{G_N}(u)}  \}.$$
We will show that for all $\vec \sigma \in \{0, 1, \ldots, q-1\}^N$  such that $\P^N_{\bm \beta}(\vec \sigma) > 0$, $|\cU^N_r(\vec \sigma)| \gtrsim N$.   This would then establish \eqref{eq:sigma_neighborhood}, since $|\cU^N_r(\vec \sigma)|=\ucons_r(\vec \sigma)$, and complete the proof of Corollary \ref{cor:coloring} due to Theorem \ref{thm:beta_H}. 

Note that if  $|\cU^N_r(\vec \sigma)|  \geq N/2$, we are trivially done.  Therefore, it suffices to assume that $|(\cU^N_r(\vec \sigma))^c| \geq N/2$, where for any set $A$, $A^c$ denotes the complement of $A$.
Now, let
  \[  S := \sup_{N \in  \N} \frac{1}{N} \sum_{v=1}^N d_{2,  v}(G_n),  \]
which is finite by assumption. 
Then by Lemma \ref{lm:2neighborhood} there exists a set $\cY \subseteq (\cU^N_r(\vec \sigma))^c$ with $|\cY| \geq \frac{N}{4(1+M)}$  such that $\cY$ is neighborhood disjoint in $G_N$. Since for every vertex in $\cY$ there is a corresponding vertex in its neighborhood that is assigned the color $r$, by the definition of the set $(\cU^N_r(\vec \sigma))^c$, and the fact that this collection of corresponding vertices are all distinct, we have, when $|(\cU^N_r(\vec\sigma))^c| \geq N/2$, 
$$|\cU^N_r(\vec \sigma)| \geq \cons_r(\vec \sigma) \geq |\cY| \geq \tfrac{N}{4(1+ S)},$$
where the first inequality uses \eqref{eq:Ur_condition}. This completes the proof of Corollary \ref{cor:coloring}. 
\end{proof}

\section{Examples where consistent estimation fails}
\label{sec:examples}

In this section we demonstrate why the sufficient conditions imposed in the consistency results for general $H$-coloring models are required, by presenting two illustrative examples where the conditions are violated and consistent estimation of the parameters is impossible. In Section \ref{subs-eg1}, we construct an $H$-coloring model where consistent estimation is impossible even when the sequence of graphs $\{G_N\}_{N \in  \N}$ has uniformly bounded maximum degree. This shows, unlike in the hardcore model, for general $H$-coloring models just assuming bounded average degree is not enough to guarantee $\sqrt N$-consistency. Next, in Section \ref{subs-eg2}, we construct a $q$-coloring model on a sequence of graphs $\{G_N\}_{N \in  \N}$ with bounded average degree where consistent estimation is impossible. This shows that, unlike in models with unconstrained vertices, bounded average degree does not ensure $\sqrt N$-consistent estimation for the $q$-coloring model, and thus some stronger condition is required.

 We begin by recalling the definition of consistency of estimates:

  \begin{defn}
    \label{def-cons-est}
    A sequence of estimates $\{\check {\bm \beta}^N\}_{N \in \N}$ is said to be {\it consistent} for the $H$ coloring model \eqref{eq:model_H_II},  if $\check {\bm \beta}^N$ converges
    in probability to $\bm \beta$, for all $\bm \beta \in \R^{q-1}$.
  \end{defn}

\begin{figure*}[!h]
\centering
\begin{minipage}[c]{1.0\textwidth}
\centering
\includegraphics[width=3.25in]
    {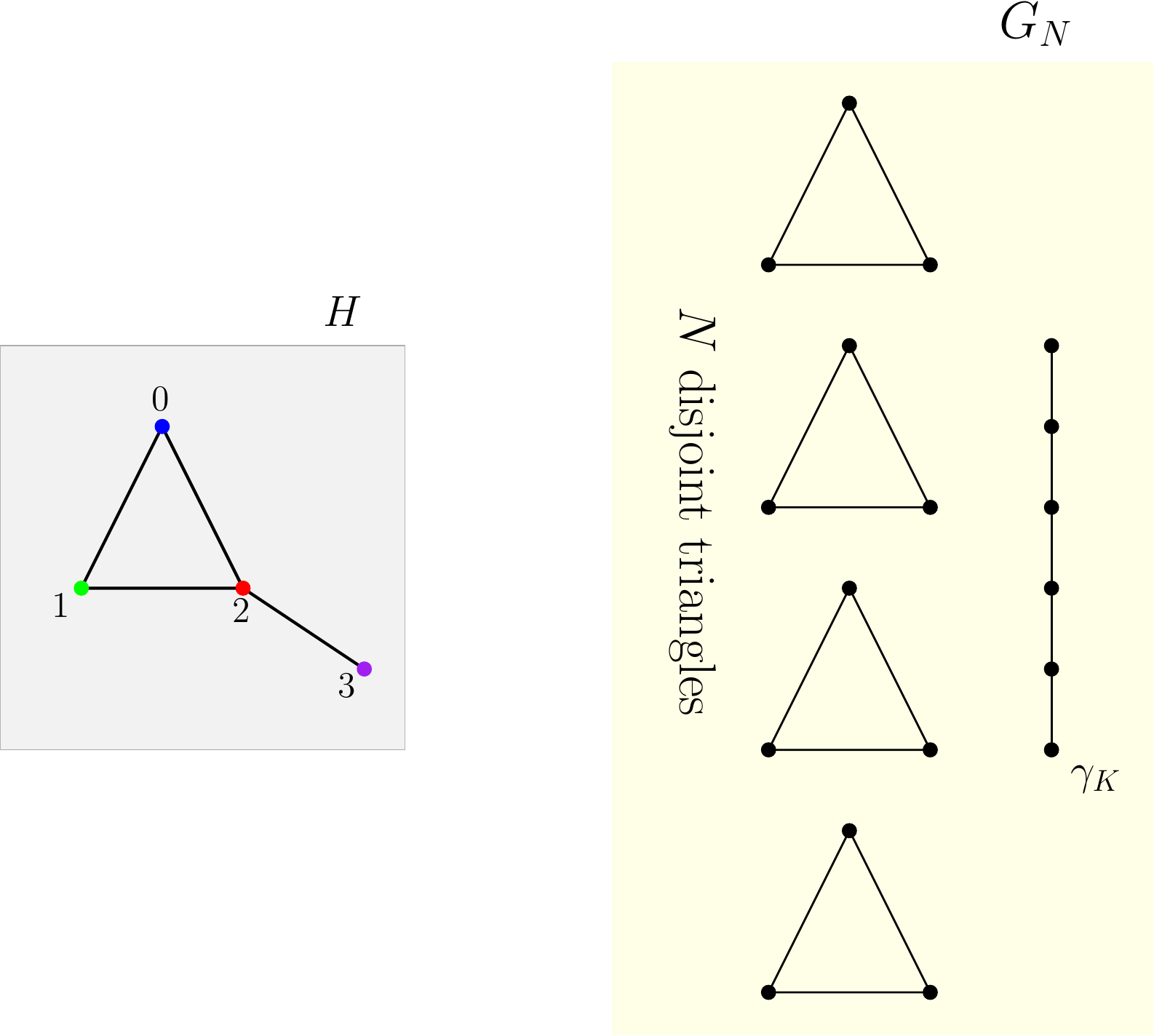}\\
\end{minipage}
\caption{\small{The $H$-coloring model corresponding to Section \ref{subs-eg1}.}} 
\label{fig:H_graph}
\end{figure*}

\subsection{Example of an $H$-coloring on a uniformly bounded degree graph sequence}
\label{subs-eg1}

Here, we construct a sequence of graphs $\{G_N\}_{N \in  \N}$ with bounded maximum degree and a constraint graph $H$ where the condition \eqref{eq:sigma_neighborhood} does not hold,
and consistent estimation is impossible. To this end, fix $K \geq 1$. Let $G_N$ be the disjoint union of $N$ triangles
and the path $\gamma_K$ with $K$ vertices (as shown in Figure \ref{fig:H_graph}). Suppose the constraint graph $H$, with $q=4$ vertices labeled $\{0, 1, 2, 3\}$, is as shown in Figure \ref{fig:H_graph}.  It is easy to see that  $\Omega_{G_N}(H)$ is non-empty. Label the vertices of the $N$ triangles $1, 2, \ldots, 3 N$ and the vertices of the path $3N+1, 3N+2, \ldots 3N+K$.  Since
$3$ is adjacent to just one color in the graph $H$, precisely, $\cN_H(3) = \{2\}$, 
 there is no valid coloring $\vec \sigma \in \Omega_{G_N}(H)$ with  $\sigma_{u}=3$, for a vertex  $1 \leq u \leq 3 N$ in one of the triangles. This implies,  $$|\{u \in V(G_N): \vec \sigma_{\cN_{G_N}(u)} = \{2\}\}| \leq K.$$
Therefore, condition \eqref{eq:sigma_neighborhood} does not hold for this model, when $r=3$ (recall that $\cN_H(3) = \{ 2 \}$).

In fact, in this case consistent estimation of $\bm \beta = (\beta_1, \beta_2, \beta_3) \in \R^3$ is impossible. To show this, assume that there exists a sequence of estimates $\{\check {\bm \beta}^N\}_{N \in \N}$ that is consistent for $\bm \beta \in \R^3$. Now, consider the following hypothesis testing problem, 
\begin{align}\label{eq:H0_beta123}
H_0: \bm \beta = (0, 0, \theta) \text{~and~} H_1: \bm \beta= (0, 0, \theta'),
\end{align}
where $\theta \ne \theta' \in \R$ is such that $|\theta-\theta'| > \delta$, for some fixed $\delta >0$. We recall the definition of consistency for tests.  
 \begin{defn}
    \label{def-cons-test}
    A sequence of test functions $\{\phi_N\}_{N \in \N}$ is said to be {\it consistent} 
    if both its Type I and Type II errors converge to zero as $N \rightarrow \infty$,
    that is, $\lim_{N\rightarrow\infty}\E_{H_0}[\phi_N]=0$, and the power satisfies 
    $\lim_{N\rightarrow\infty}\E_{H_1}[\phi_N]=1$.
  \end{defn}

Using the sequence of consistent estimates $\{\check {\bm \beta}^N\}_{N \in \N}$ we can now construct a consistent sequence of test functions $\{\phi_N\}_{N \in \N}$ for the hypothesis \eqref{eq:H0_beta123} as follows: 
$$\phi_N= \bm 1\{ \check {\bm \beta}^N \notin B_{\theta, \theta'}\},$$
where $B_{\theta, \theta'}$ is the ball in $\R^3$ with center at $(0, 0, \theta) $ and radius $|\theta-\theta'|/2$. This means, in other words, that the corresponding test will reject the null hypothesis $H_0$ whenever $\check {\bm \beta}^N \notin B_{\theta, \theta'}$.  Hereafter, for $\theta, \theta' \in \R$ denote $\bm \theta := (0, 0, \theta)$ and $\bm \theta' := (0, 0, \theta')$, and let $\P^N_{\bm \theta}$ be the probability measure of the $H$-coloring model \eqref{eq:model_H_II}, with $H$ as in Figure \ref{fig:H_graph} and parameters $(\beta_1, \beta_2, \beta_3)= (0, 0, \theta)$.   We now compute the Kullback-Leibler (KL) divergence between the  probability measures $\P^N_{\bm \theta}$ and $\P^N_{\bm \theta'}$. By definition, 
\begin{align}\label{eq:theta12}
 D_{\mathrm{KL}}(\P^N_{\bm \theta}||\P^N_{\bm \theta'}) :=\E_{\vec \sigma \sim \P^N_{\bm \theta}}\left[\log\frac{\P^N_{\bm \theta}(\vec \sigma)}{\P^N_{\bm \theta'}(\vec \sigma)} \right] = F_{G_N}(\bm \theta')-F_{G_N}(\bm \theta) -(\theta'-\theta)\frac{\partial}{\partial \beta_3}F_N(\bm \beta)\Big|_{ \bm \beta = (0, 0, \theta)}. 
\end{align}
Clearly, recalling the definition of $\cons_3(\vec \sigma)$ from \eqref{eq:count_sigma}, 
$$\frac{\partial}{\partial \beta_3}F_N(\bm \beta)\Big|_{ \bm \beta = (0, 0, \theta)}= \E_{ \vec \sigma \sim \P^N_{\bm \theta}} \left[ \cons_3(\vec \sigma) \right] \leq K = O(1),$$
since, as discussed above, no vertex of the $N$ disjoint triangles can be assigned color 3.
Combining this with Lemma \ref{lm:theta12} below shows, $$D_{\mathrm{KL}}(\P^N_{\bm \theta}||\P^N_{\bm \theta'})= O(1).$$ 
 This implies, by \cite[Proposition 6.1]{BM16}, that there is no sequence of consistent tests for \eqref{eq:H0_beta123}, which leads to a contradiction. In turn, this  shows that there is  no sequence of consistent estimates for $\bm \beta \in \R^3$ for the $H$-coloring model  in Figure \ref{fig:H_graph}.

\begin{lem}\label{lm:theta12} Let $H$ be as in Figure \ref{fig:H_graph}. Then, for any $\theta \ne \theta' \in \R$, $$F_{G_N}(\bm \theta') - F_{G_N}(\bm \theta) = O_{\theta, \theta', K}(1).$$
\end{lem}

\begin{proof} Fix $\theta \in \R$. Then using the bounds $0\leq \cons_3(\vec \sigma) \leq K$ gives,  
\begin{align*}
\min\{ e^{\theta K}, 1 \} |\Omega_{G_N}(H)|  \leq \sum_{\vec \sigma \in \Omega_{G_N}(H)}  e^{\theta \cons_3(\vec \sigma) }  \leq \max\{ e^{\theta K}, 1 \} |\Omega_{G_N}(H)|.  
\end{align*} 
Therefore, taking logarithms gives, 
\begin{align*}
\log \min\{ e^{\theta K}, 1 \} & + \log |\Omega_{G_N}(H)|  \leq F_{G_N}(\bm \theta)  \leq \log \max\{ e^{\theta K}, 1 \} + \log |\Omega_{G_N}(H)|.  
\end{align*} 
This implies, for $\theta \ne \theta' \in \R$, 
\begin{align*}
| F_{G_N}(\bm \theta') - F_{G_N}(\bm \theta) | = O_{\theta, \theta', K}(1), 
\end{align*} 
which proves the lemma.  
\end{proof}

\subsection{Example of a $q$-coloring model on a bounded average degree graph}
\label{subs-eg2}

Here, we construct an example where the graph sequence $G_N$ has bounded average degree, but consistent estimation is impossible in the $q$-coloring model on $G_N$. To this end, fix $q \geq 2$ and construct the graph $G_N= (V(G_N), E(G_N))$ as follows: Let $V(G_N)= V_1 \bigcup V_2$ with $V_1:=\{u_1, \ldots, u_{q-1}\}$ and $V_2 := \{w_1, \ldots, w_N\}$, and $E(G_N)= E_1 \bigcup E_2$, with $E_1=\{(u_i, u_j): 1 \leq i < j \leq q-1\}$ and $E_2 := \{ (u_i, w_j): 1 \leq i \leq q-1 \text{ and } 1 \leq j \leq N \}$. In order words, $G_N$ has a clique of size $q-1$ on the vertex set $V_1$ and a complete bipartite graph between the sets $V_1$ and $V_2$, where $|V_1|=q-1$ and $|V_2|=N$. Figure \ref{fig:coloring_graph_example} shows the graph $G_5$ when $q=3$.

\begin{figure*}[!h]
\centering
\begin{minipage}[c]{1.0\textwidth}
\centering
\includegraphics[width=3.05in]
    {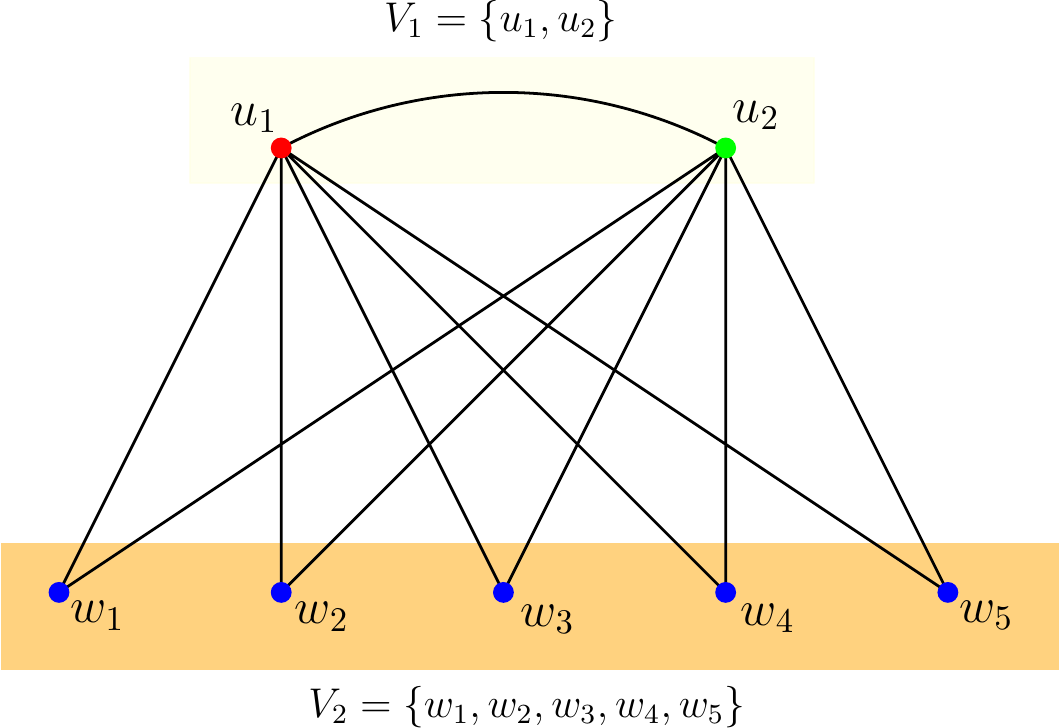}\\
\end{minipage}
\caption{\small{The graph corresponding to Section \ref{subs-eg2} when $q=3$ and $N=5$.}}
\label{fig:coloring_graph_example}
\end{figure*}

In any valid coloring of $G_N$ all the vertices of $V_1$ must have distinct colors and, hence, all the vertices in $V_2$ must have the same (remaining) color. For $0 \leq r \leq q-1$, denote by $\Omega_{K_q, r}(G_N)$ the set of  valid colorings of $G_N$ where every vertex in $V_2$ has color $r$. Clearly, the sets $\Omega_{K_q, 0}(G_N),  \ldots, \Omega_{K_q, q-1}(G_N)$ form a disjoint partition of the set of all valid $q$-colorings of $G_N$. Note that $|\Omega_{K_q, r}(G_N)|=(q-1)!$, for all $1 \leq r \leq q-1$. Then setting $\beta_0=0$ it follows that, for $\vec \tau \in \Omega_{K_q, r}(G_N)$, 
\begin{align*} 
 \P_\beta^N(\vec \sigma=\vec \tau)  =\frac{\exp\left\{\beta_r N + \sum_{s\in \{0, 1, \ldots, q-1\}\backslash\{r\} } \beta_s  \right\} }{(q-1)! \sum_{a=0}^{q-1}  \exp\left\{\beta_a N + \sum_{s\in \{0, 1, \ldots, q-1\}\backslash\{a\} } \beta_s \right\} }. 
\end{align*} 
Note that when the constraint graph $H=K_q$, the set of $r$-unconstrained vertices is $\cU^N_r(\vec \sigma):= \{ u \in \{1, \ldots, N\}: r \notin \vec \sigma_{\cN_{G_N}(u)}  \}$ (recall \eqref{eq:unconstrained_r}). Therefore, for every $\varepsilon \in (0, 1)$ fixed, 
\begin{align*}
 \P_\beta^N(|\cU_1^N(\vec \sigma)| > \varepsilon N)  & =  \P_\beta^N(|\cU_1^N(\vec \sigma)|=N) \nonumber \\ 
& = \sum_{\vec \tau \in  \Omega_{K_q, 1}(G_N)}\P_\beta^N(\vec \sigma=\vec \tau) \nonumber \\ 
& = \frac{\exp\left\{\beta_1 N + \sum_{s\in \{0, 1, \ldots, q-1\}\backslash\{1\} } \beta_s  \right\} }{\sum_{a=0}^{q-1}  \exp\left\{\beta_a N + \sum_{s\in \{0, 1, \ldots, q-1\}\backslash\{a\} } \beta_s \right\} }  \nonumber \\ 
& = \frac{\exp\left\{ \sum_{s\in \{0, 1, \ldots, q-1\}\backslash\{1\} } \beta_s  \right\} }{\sum_{a=0}^{q-1}  \exp\left\{(\beta_a - \beta_1) N + \sum_{s\in \{0, 1, \ldots, q-1\}\backslash\{a\} } \beta_s  \right\} }. 
\end{align*}
Now, taking limits on both sides as $N \rightarrow \infty$ gives, 
\begin{align}\label{eq:condition_q_example} 
\limsup_{N \rightarrow \infty}  \P_\beta^N(|\cU_1^N(\vec \sigma)| > \varepsilon N) = 0,
\end{align} 
whenever $\beta_a > \beta_1$, for some $a \in \{0, \ldots, q-1\} \backslash\{1\}$. In particular, \eqref{eq:condition_q_example} holds whenever $\beta_1 < 0$, since $\beta_0=0$, irrespective of the values of $\beta_r$, for $2 \leq r \leq q-1$.  This implies, for $\beta_1 < 0$, 
\begin{align*}
 \limsup_{N \rightarrow \infty} \P^N_{\bm \beta}\left( |\cU_r^N(\vec \sigma)|   > \varepsilon N, \text{ for all } 1 \leq r \leq q-1 \right) 
& \leq \limsup_{N \rightarrow \infty} \P^N_{\bm \beta}\left( |\cU_1^N(\vec \sigma)|   > \varepsilon N \right) = 0,   
\end{align*} 
that is, recalling $\ucons_1(\vec \sigma)=|\cU_1^N(\vec \sigma)|$, \eqref{eq:sigma_neighborhood} does not hold. 

In this case, as in the example above, we can again show consistent estimation of $\bm \beta$ is impossible when $\beta_1 < 0$. To this end, for $\theta \in \R$, let $\bm \theta := ( \theta, 0, \ldots, 0) \in \R^{q-1}$ and let $\P^N_{\bm \theta}$ be the probability measure of the $q$-coloring model with parameter $\bm \theta$.  
To begin with, note that 
\begin{align*}
\lim_{N \rightarrow \infty} F_{G_N}(\bm \theta) & = \lim_{N \rightarrow \infty}  \log Z_{G_N}(\bm \theta) \nonumber \\  
& = \lim_{N \rightarrow \infty}  \log \left\{ (q-1)! \left(  e^{\theta N}  + (q-1) e^{\theta} \right) \right \} \nonumber \\ 
& =  \log \left\{ (q-1) (q-1)!  e^{\theta} \right \},
\end{align*}
for $\theta < 0$. This shows that, for $\theta < 0$, $F_{G_N}(\bm \theta) = O(1)$. Therefore, for $\theta < \theta' < 0$,  by monotonicity,  
To begin with, note that 
\begin{align*}
 (\theta'-\theta) \frac{\partial}{\partial t}F_{G_N}((t, 0, \ldots, 0))\Big|_{t = \theta}  & \leq \int_{\theta}^{\theta'}  \frac{\partial}{\partial t}F_{G_N}((t, 0, \ldots, 0)) \mathrm d t \nonumber \\ 
& = F_{G_N}(\bm \theta') -  F_{G_N}(\bm \theta) \nonumber \\ 
& = O(1). 
\end{align*} 
Therefore, the KL divergence between the  probability measures $\P^N_{\bm \theta}$ and $\P^N_{\bm \theta'}$, for $\theta < \theta' < 0$ is, 
\begin{align*}
 D_{\mathrm{KL}}(\P^N_{\bm \theta}||\P^N_{\bm \theta'}) 
& = F_{G_N}(\bm \theta')-F_{G_N}(\bm \theta)-(\theta'-\theta) \frac{\partial}{\partial t}F_{G_N}((t, 0, \ldots, 0))\Big|_{t = \theta}  \nonumber \\  
& = O(1).  
\end{align*}
This implies, by the same arguments as those used in Section \ref{subs-eg1}, that consistent estimation of $\bm \beta$ in the $q$-coloring model on $G_N$ is impossible whenever $\beta_1 < 0$. In fact, by the symmetry of the colors in the $q$-coloring model, this implies consistent estimation is impossible whenever there exists some $1 \leq r \leq q-1$ for which $\beta_r < 0$.  \\

\small{\subsection*{Acknowledgment}  
B. B. Bhattacharya was supported in part by NSF CAREER Grant DMS-2046393. K. Ramanan was supported in part by ARO Grant W911NF2010133. The authors also thank the anonymous referees for their valuable comments which greatly improved the quality and the presentation of the paper. }

\end{document}